\documentclass[11pt]{amsart}
\usepackage{amscd, amssymb}
\usepackage{tikz}
\usetikzlibrary{arrows,positioning} 

\usepackage{comment}

\usepackage{amsthm}
\usepackage{mathrsfs}
    \usepackage{bbm}
\usepackage{ stmaryrd }

\usepackage{multirow}
\usepackage{makecell}

\usepackage{hyperref}
\hypersetup{
	colorlinks=true,
	linkcolor=blue,
	filecolor=magenta,  
	urlcolor=cyan,
}

\def\End{\operatorname{End}}

\def\ker{\operatorname{ker}}

\def\dim{\operatorname{dim}}
\def\ad{\mathsf{ad}}

\def\id{\operatorname{id}}
\def\Dist{\operatorname{Dist}}

\def\Ind{\operatorname{Ind}}

\def\Hom{\operatorname{Hom}}

\newcommand{\uHom}{\underline{\mathrm{Hom}}}
\newcommand{\uEnd}{\underline{\mathrm{End}}}

\def\R{\mathbb{R}}
\def\N{\mathbb{N}}

\def\UU{\mathcal{U}}

\def\JJ{\mathcal{J}}

\def\Cs{\mathscr{C}}
\def\Ds{\mathscr{D}}
\def\Es{\mathscr{E}}

\def\m{\mathfrak{m}}

\def\cA{\mathscr{A}}
\def\cC{\mathscr{C}}

\def\bk{\Bbbk}
\def\mR{\mathbb{R}}
\def\mZ{\mathbb{Z}}
\def\mN{\mathbb{N}}
\def\gd{\mathsf{gd}}

\def\mod{\mathrm{ mod}}
\def\Rep{\operatorname{Rep}}

\def\Spec{\operatorname{Spec}}

\def\Vec{\operatorname{Vec}}
\def\VEC{\operatorname{VEC}}
\def\sVec{\operatorname{sVec}}
\def\Ver{\operatorname{Ver}}
\def\gr{\operatorname{gr}}
\def\Ext{\mathrm{Ext}}
\def\Tor{\mathrm{Tor}}

\def\Set{\text{Set}}

\def\Mod{\mathrm{Mod}}
\def\sub{\subseteq}

\def\Ab{\text{Ab}}

\def\cB{\mathscr{B}}

\newcommand{\mL}{\mathbb{L}}

\def\uRep{\underline{\mathrm{Rep}}}

\def\xto{\xrightarrow}

\def\onto{\twoheadrightarrow}

\newtheorem{thm}{Theorem}[section]
\newtheorem{cor}[thm]{Corollary}
\newtheorem{lemma}[thm]{Lemma}
\newtheorem{prop}[thm]{Proposition}
\newtheorem{question}[thm]{Question}

\theoremstyle{definition}
\newtheorem{definition}[thm]{Definition}

\theoremstyle{remark}
\newtheorem{remark}[thm]{Remark}
\newtheorem{example}[thm]{Example}

\numberwithin{equation}{section}

\usepackage{relsize}
\usepackage{fancyhdr}
\usepackage[all,cmtip]{xy}
\newcommand{\unit}{\mathbf{1}}
\newcommand{\mC}{\mathbb{C}}
\newcommand{\uExt}{\underline{\Ext}}
\newcommand{\mG}{\mathbb{G}}
\newcommand{\ev}{\mathrm{ev}}
\newcommand{\Isom}{\mathcal{I}som}
\newcommand{\CAlg}{\mathrm{CAlg}}
\begin{document}

\title{Absolutely flat algebras in tensor categories}
		\author{Kevin Coulembier, Alexander Sherman}

        \subjclass[2020]{}

	\begin{abstract}
		We study the relation between semisimple algebras, artinian simple algebras and artinian absolutely flat algebras in ind-completions of tensor categories (rigid abelian monoidal categories). Our main results show that the three types of algebras coincide, if we restrict to commutative algebras in certain (all, if we accept the main conjectures in the field) symmetric tensor categories of moderate growth. We also provide examples to show that in general these notions tend to diverge, contrary to the classical case of algebras over a field. In one appendix we give the first, to the best of our knowledge, example of a tensor category with finite-dimensional morphism spaces but objects of infinite length. In a second appendix we establish a generalisation of Deligne's theorem that tannakian categories are neutral over algebraically closed fields.
	\end{abstract}

	\maketitle
	\pagestyle{plain}
\section{Introduction}
By basic ring theory, an artinian algebra over a field $\Bbbk$ is semisimple if and only if it is a product of simple algebras, and if and only if it is absolutely flat (von Neumann regular). 
Following \cite{EGNO}, a tensor category is an abelian $\bk$-linear rigid monoidal category satisfying some finiteness properties, with the most basic example being the category of finite dimensional vector spaces over $\bk$.  It is thus a logical question whether the above properties on artinian $\Bbbk$-algebras remain equivalent in the ind-completion of any tensor category.

This is a natural question in its own right, but has also already appeared in several special cases more directly motivated in the literature, perhaps with different formulation and terminology. For example the open question, answered in the (super) tannakian case in \cite{Del14, CEOP}, of whether the extension of scalars of a tensor category is always a tensor category, is equivalent to asking whether field extensions are absolutely flat as algebras in tensor categories. An affirmative answer to the latter question would also answer Deligne's question from \cite{Del90} of whether Deligne tensor products of pretannakian categories are always pretannakian. In \cite{EO}, Etingof and Ostrik conjectured that (finite) simple algebras in finite tensor categories are `exact', which is equivalent to being absolutely flat. This conjecture was proved to be correct in \cite{CSZ}. These questions also appear naturally, see for instance \cite{CEO}, in the context of the ongoing exploration of the structure theory of pretannakian categories (of moderate growth in positive characteristic, and of superexponential growth in general). For other recent work in which such questions appear we refer to \cite{Le, SY}.

In this paper, we take a systematic approach to the question of whether, for a given tensor category, an artinian algebra is semisimple if and only if it is a product of simple algebras, or if and only if it is absolutely flat. Most of our results answer parts of this question in the affirmative, either in general or in special cases. Nonetheless some results demonstrate that in complete generality the answer deviates from the classical case.

More concretely, in Section~\ref{sec:genstuff} we show that some implications survive, see Theorems~\ref{prop:consolidate} and~\ref{prop:consolidate2}, and that the main open questions are whether semisimple algebras need to be products of simple algebras and whether artinian simple algebras need to be absolutely flat. With respect to answering these remaining questions, we will primarily focus on  the specific case of commutative algebras in symmetric (i.e. pretannakian) tensor categories.

In Section~\ref{sec:examples} we review the special case of finite algebras in finite tensor categories which was resolved in \cite{CSZ}, discuss the open problem of algebras that are field extensions of $\bk$, and finally produce two counterexamples. Firstly, we show that there exists a semisimple commutative algebra in a pretannakian category (over the field of complex numbers, or any other field) that is not a product of simple algebras. Before stating the second counterexample, we recall that simple commutative $\bk$-algebras are obviously automatically artinian. In contrast, for the proper notion of artinian algebras in tensor categories, this is no longer the case. We show that there exists an example of a simple commutative algebra in the ind-completion of a pretannakian category that is not absolutely flat. Unfortunately, since the result is not constructive we do not know whether this is an example of a non-artinian simple commutative algebra, or of an artinian simple commutative algebra that is not absolutely flat.

Both examples above are in pretannakian categories that are not of `moderate growth'. Our main results indicate that, in contrast, in pretannakian categories of moderate growth, the semisimple algebras satisfy their classical properties. Conjecturally \cite{BEO, C1}, every such category admits a symmetric tensor functor to a pretannakian category that is geometrically reductive (GR) and maximally nilpotent (MN) as defined in \cite{C1}. As symmetric tensor functors to such well-behaved basic pretannakian categories are often called fibre functors, we call a pretannakian category admitting such a tensor functor \emph{well-fibered}. By \cite{Del90, Del02} every pretannakian category of moderate growth in characteristic 0 is well-fibered, and by \cite{CEO-A}, in positive characteristic every Frobenius exact pretannakian category of moderate growth is well-fibered. More generally, as mentioned above, every pretannakian category of moderate growth is expected to be well-fibered. The following will be proved as Theorems~\ref{thm:main} and~\ref{thm noncomm case}.

\medskip

{\bf Theorem A:} {\em Let $\cC$ be a well-fibered pretannakian category over an algebraically closed field~$\bk$. Then
\begin{enumerate}
    \item Every simple commutative algebra in $\Ind\cC$ is artinian and absolutely flat.
    \item Every semisimple commutative algebra in $\Ind\cC$ is a product of simple algebras.
    \item If $\mathrm{char}(\bk)>0$, then every semisimple algebra in $\cC$ is a product of simple algebras, and absolutely flat.
\end{enumerate}}

\medskip

As already mentioned, the main result of \cite{CEO-A} implies that the above theorem applies to Frobenius exact pretannakian categories of moderate growth. Similarly, by \cite[Thm. 9.2.1 and 9.3.7]{CEO} we can conclude the same for pretannakian categories $\Cs$ that fiber over $\Ver_{2^\infty}$.
Parts (2) and (3) are proved in Section~\ref{sec:modgr1} and part (1) is proved in Section~\ref{sec:simpleGRMN}. The results in Section~\ref{sec:simpleGRMN} actually give a more explicit description of simple commutative ind-algebras. To avoid technicalities, here we state the result for tannakian categories only, as the results seem not to be well-known even in this generality. The following will be proved in Theorems~\ref{thm:main}, \ref{thm noncomm case} and \ref{thm simple comm alg}.

\medskip

{\bf Theorem B:} {\em Let $G$ be an affine group scheme over an algebraically closed field $\bk$.
\begin{enumerate}
    \item Let $A$ be a simple commutative $G$-equivariant $\bk$-algebra. Let $\mathbb{L}$ be the algebraic closure of the field $A^G$ of $G$-invariants in $A$. Then there exists an exact subgroup $H$ of $G_{\mathbb{L}}$ such that $A\otimes_{A^G}\mathbb{L}\cong \mathbb{L}[G_{\mathbb{L}}]^H\cong \mathbb{L}[G_{\mathbb{L}}/H]$.
    \item Let $A$ be a semisimple commutative $G$-equivariant $\bk$-algebra, then $A$ is a product of simple such algebras.
    \item Let $A$ be a finite dimensional semisimple $G$-equivariant $\bk$-algebra, then $A$ is a product of simple such algebras.
\end{enumerate}}

\medskip

Note that Theorem~B(1) for $\mathrm{char}(\bk)=0$ and $G$
 of finite type was proved by Magid in \cite{Magid}. Furthermore, since Theorem~B(3) does not require positive characteristic, contrary to Theorem~A(3), we leave open the question whether Theorem~B(3) remains true for supergroups in characteristic zero.

In Appendix~\ref{App} we give an example of a symmetric tensor category with finite dimensional morphism spaces that is not pretannakian (has objects of infinite length). 
This serves to highlight subtle potential properties of module categories of simple commutative algebras in pretannakian categories. 

In Appendix~\ref{app:Del} we generalise Deligne's result that tannakian categories over an algebraically closed field are neutral, see \cite{Del02, Duke}, from $\Vec$ to arbitrary GR+MN pretannakian categories. This is used in the main part of the paper to get a tighter description of the field extensions of $\bk$ one needs to consider to describe simple algebras over affine group schemes in such categories.

In Appendix~\ref{App:C} we prove a technical result that we use to fill in a minor gap in the literature surrounding extensions of scalars of tensor categories.

\section{Preliminaries}\label{sec:prel}

Unless further specified, $\bk$ is an arbitrary field.

\subsection{Tensor categories}\label{prel:tc}

An essentially small $\bk$-linear monoidal category $(\cC,\otimes,\unit)$ with $\bk$-linear tensor product is a {\bf tensor category over $\bk$} if
\begin{enumerate}
\item $\cC$ is abelian;
\item $\bk\to\End(\unit)$ is an isomorphism;
\item $(\cC,\otimes,\unit)$ is rigid, meaning that every object $X$ has a left monoidal dual $X^\ast$ and right monoidal dual ${}^\ast X$.
\end{enumerate}
By \cite[Proposition~1.17]{DM}, it then follows that $\unit$ is simple. If additionally, we have
\begin{enumerate}
\item[(4)] every object $X$ in $\cC$ has finite length $\ell(X)$,
\end{enumerate}
then $(\cC,\otimes,\unit)$ is an {\bf artinian tensor category}.
Using (2) and (3) shows also that morphism spaces are finite dimensional in an artinian tensor category. A {\bf symmetric tensor category} is a tensor category with a symmetric braiding $\sigma$. A symmetric artinian tensor category is called a {\bf pretannakian} category. The (categorical) dimension $\dim X\in\bk$ of an object $X$ in a pretannakian category is $\ev_X\circ \sigma_{X,X^\ast}\circ\mathrm{co}_X$, where $\ev_X:X^\ast\otimes X\to\unit$ is the evaluation and $\mathrm{co}_X:\unit\to X\otimes X^\ast$ the coevaluation.

The first example of pretannakian category is the category of finite dimensional vector spaces $\Vec=\Vec_{\bk}$. We denote the category of all vector spaces over $\bk$ by $\VEC=\VEC_{\bk}$. More generally, for an affine group scheme $G$ over $\bk$, its category of finite-dimensional rational representations is pretannakian.  Another example is  the pretannakian category $\sVec$ of super vector spaces, which is the category of finite dimensional $\mZ/2$-graded vector spaces with Koszul sign rule for the symmetric braiding. The non-trivial simple object in $\sVec$ is the odd line $\bar{\unit}$.

A {\bf tensor functor} between tensor categories over $\bk$ is a $\bk$-linear exact monoidal functor. Here, we can replace exact by faithful, see \cite[Theorem~2.4.1]{CEOP}. A tensor functor $F:\cC\to\Ds$ always has a right adjoint $F_\ast:\Ind\Ds\to\Ind\cC$ between the ind-completions. By construction, $F_\ast$ has a lax monoidal structure.
Following \cite{Del90}, for a tensor category $\cC$ we write $\Gamma=\Hom(\unit,-):\Ind\cC\to\VEC$. In other words, $\Gamma=F_\ast$ for the canonical inclusion $F:\Vec\hookrightarrow\cC$. We will occasionally also refer to a $\bk$-linear exact monoidal functor $\Cs\to\Ds$, where
$\Cs$ is a tensor category over $\bk$ and $\Ds$ a tensor category over a field extension $K$ of $\bk$, as a tensor functor.

For $\cC$ a tensor category, we have the internal Hom functor $\uHom$ on $\Ind\cC$ satisfying
\[
\Hom(X\otimes Y,Z)\;\cong\;\Hom(Y,\uHom(X,Z)).
\]
Of course, if $X\in \cC$, we simply have $\uHom(X,Z)\cong  {^*}X\otimes Z$.

\begin{remark}\label{rem:uExt} Let $\cC$ be a tensor category with $X\in\Ind\cC$. If $X\in\cC$, then clearly $\uHom(X,-):\cC\to\cC$ is exact. In general, $\uHom(X,-):\Ind\cC\to\Ind\cC$ is only left exact and we denote the right derived functors by $\uExt^i(X,-)$. For example, let $\cC$ be the rational representation category of the additive group $\mG_a$ over the complex numbers. Then with $I\cong \bk[\mG_a]$ the injective hull of the trivial representation $\unit$, we have $$\uHom(I,\unit)\;=\;\varprojlim_i M_i^\ast\;=\;0,$$
    where $M_i\cong M_i^\ast$ is the length $i$ submodule of $I$. Moreover, if we apply $\uHom(I,-)$ to the injective resolution $\unit\hookrightarrow I\twoheadrightarrow I$, then the endomorphism of $\uHom(I,I)$ in
    \[0\to\uHom(I,I)\to \uHom(I,I)\to 0\]
is not an epimorphism (note we already saw it is a monomorphism), so $\uExt^1(I,\unit)\not=0$. Indeed, one can verify directly that the canonical $\unit\subset\uHom(I,I)$ is not in the image.    
\end{remark}

A {\bf finite} tensor category is an artinian tensor category with enough projective objects and finitely many isomorphism classes of simple objects. We say that an artinian tensor category $\cC$ has {\bf moderate growth} if for each $X\in\cC$ there exists $C\in\mR$ such that the length of the tensor powers of $X$ are bounded as $\ell(X^{\otimes n})\le C^n$, see \cite{CEO-A}. This is equivalent to demanding that the quantity
\begin{equation}\label{def:gd}\gd(X)\;:=\; \lim_{n\to\infty}\sqrt[n]{\ell(X^{\otimes n})}\;\in\;\mR_{\ge 0}\cup\{\infty\}\end{equation}
is finite for each $X\in\cC$. Finite tensor categories are of moderate growth, since $\gd(X)$ is the Frobenius-Perron dimension, see \cite[Chapter 4]{EGNO}.

Assume that $\bk$ is a perfect field. For two pretannakian categories $\cC,\Ds$, we can consider the
Deligne tensor product $\cC\boxtimes\Ds$, see for instance \cite{Del90, CEOP}. It is again a pretannakian category, equipped with symmetric tensor functors 
\begin{equation}\label{eq:delTP}\cC\xrightarrow{X\mapsto X\boxtimes \unit}\cC\boxtimes\Ds\xleftarrow{\unit\boxtimes Y \mapsfrom Y} \Ds,\quad \end{equation}
combining into a bifunctor $(X,Y)\mapsto X\boxtimes Y$. The universal property of the product states that a symmetric tensor functor $\cC\boxtimes \Ds\to\mathscr{E}$ corresponds to a pair of symmetric tensor functors $\cC\to\mathscr{E}\leftarrow\Ds$, via composition with the tensor functors in \eqref{eq:delTP}.

\subsection{ Algebras and modules}\label{sec:comalg1}

Let $\cC$ be a tensor category over $\bk$. We will refer to a monoid object $(A,m,\eta)$ in $\Ind\cC$ as an {\bf ind-algebra} in $\cC$, or simply an algebra in $\cC$ if $A\in\cC$. For example, an algebra in $\VEC$ is the same as a $\Bbbk$-algebra.

\begin{lemma}[Chinese reminder theorem]\label{lemma CRT}
    Let $\mathscr{C}$ be a tensor category over $\Bbbk$, and let $A$ be an ind-algebra in $\mathscr{C}$.  If $I,J$ are ideals of $A$ such that $I+J=A$, then we have an isomorphism of algebras
    \[
    A/(I\cap J)\cong A/I\times A/J.
    \]
\end{lemma}
\begin{proof}
    The natural algebra morphism $\phi:A/(I\cap J)\to A/I\times A/J$ is clearly a monomorphism in $\Ind\cC$.  Since $I+J=A$, we have $\phi(I)=0\times A/J$ and $\phi(J)=A/I\times 0$ which shows $\phi$ is an epimorphism.
\end{proof}

For an ind-algebra $A$, we have the Grothendieck category $\Mod_A\Cs$ of left $A$-modules in $\Ind\cC$.  The category $\Mod_{A}\Cs$ is naturally a right module category over $\Cs$ via tensor product in $\Cs$.  For $A$-modules $M,N$ we can define the equaliser
\[\Ind\cC\,\ni\,\uHom_A(M,N)\to\left(\uHom(M,N)\rightrightarrows \uHom(A\otimes M,N)\right).\]
We can equivalently define this internal Hom functor via
\[
\Hom_A(M\otimes X,N)\;\cong\;\Hom(X,\uHom_A(M,N)),
\]
for arbitrary $X\in\Ind\cC$.
By construction $\uHom_A$ is left exact in both variables, and we define the derived functors $\uExt^i_A(M,N)$ as the cohomology functors of the complex $\uHom_A(M,I^\bullet)$ for an injective coresolution $I^\bullet$ of $N$. For a right $A$-module $M$ and a left $A$-module $N$, we also define the tensor product via the usual coequaliser
\[\left(M\otimes A\otimes N\rightrightarrows M\otimes N\right)\to M\otimes_A N.\]
A left $A$-module $N$ is called {\bf flat} if $-\otimes_AN$ is exact.

For the remainder of the preliminaries, we assume that $\cC$ is pretannakian. An ind-algebra $(A,m,\eta)$ is then commutative if $m\circ\sigma_{A,A}=m$, for $\sigma$ the braiding. The universal symmetric algebra generated by an object $X\in\cC$ is its symmetric algebra $S(X)=\oplus_{n\ge 0}S^n(X)$. In other words, $S^n X$ is the maximal quotient of $X^{\otimes n}$ that is invariant under the braiding action. We denote the category of commutative ind-algebras by $\CAlg\cC$ and, for a fixed commutative algebra $A$, we have the category $\CAlg_A\cC$ of $A$-algebras. 

For a commutative ind-algebra $A$, we have the category $\Mod_A{\cC}$ of all $A$-modules in $\Ind\cC$ with its
full subcategory $\mod_A{\cC}$ of finitely presented  modules.  We simply write $\Mod_A$ and $\mod_A$ when $\cC$ is clear. These categories are monoidal, with tensor product $\otimes_A$.

We conclude the preliminaries with some straightforward technical statements we will rely on. Recall that $\cC$ is a pretannakian category over a field $\bk$.

\begin{definition} For $A\in\CAlg\cC$,
    we say that \[f\in \Gamma(A)=\Hom(\unit,A)\cong\End_A(A)\] is a zero divisor in $A$ (as opposed to `in $\Gamma(A)$') if there exists $0\not=X\sub A$ with $fX=0$ and a non-zero divisor otherwise.  Equivalently, $f$ is a non-zero divisor if $A\xto{f}A$ is injective.
\end{definition}

   We write $Q(A)$ for the localization of $A$ at the multiplicative system of non-zero divisors lying in $\Gamma(A)$.  Note that the localization map $A\to Q(A)$ is a monomorphism. 

\begin{lemma}\label{lem flat ordinary module}Assume that $\bk$ is algebraically closed, $R$ a commutative $\bk$-algebra, and $M\in\Mod_R\cC$. Then $M$ is flat in $\Mod_R\cC$ if and only if $-\otimes_RM$ is exact when evaluated on $R$-modules in $\VEC\subset \Ind\cC$.
\end{lemma}
\begin{proof}
    One direction of the claim is trivial. Assume that $-\otimes_RM$ is exact when evaluated on ordinary $R$-modules. We need to prove that $\Tor_1^R(-,M)=0$ as a endofunctor on $\Mod_R\cC$. It suffices to prove vanishing on finitely generated modules so that one verifies it is sufficient to prove vanishing on quotients of $R\otimes V$ for simple $V\in\cC$. Since $\End(V)=\bk$ it follows quickly that the only such quotients are of the form $N\otimes V$ for a quotient $N$ of $R$. But now 
    \[\Tor_1^R(N\otimes V,M)\;\cong\;V\otimes \Tor_1^R(N,M)\;=\;0,\]
    which concludes the proof.
\end{proof}

\begin{cor}\label{lem:Salpha}Assume that $\bk$ is algebraically closed.
    Let $S$ be a $\bk$-algebra considered as a direct limit $S=\varinjlim_\alpha S_\alpha$ of $\bk$-algebras. If for a finitely presented $S_\alpha$-module $M_\alpha$ in $\Ind\cC$, the $S$-module $M:=S\otimes_{S_\alpha}M_\alpha$ is flat, then $S_\beta\otimes_{S_\alpha}M_\alpha$ is flat over $S_\beta$ for some $\beta>\alpha$. 
\end{cor}
\begin{proof}
    By Lemma~\ref{lem flat ordinary module}, it suffices to show that $S_\beta\otimes_{S_\alpha}M_\alpha$ is flat `as an ordinary $S_\beta$-module'. By the latter we mean that we consider a $\bk$-linear exact cocontinuous faithful (non-monoidal) functor $\omega:\Ind\cC\to\VEC$, see \cite[Theorem~1.9.15]{EGNO}, and it suffices to show that $\omega(S_\beta\otimes_{S_\alpha}M_\alpha)\cong S_\beta\otimes_{S_\alpha}\omega(M_\alpha)$ is flat.  The latter is a classical fact, see \cite[11.2.6.1(ii)]{EGA}.
\end{proof}


\section{Semisimple algebras in tensor categories}\label{sec:genstuff}

\subsection{Classical ring theory}

Let $R$ be a ring (unital, not necessarily commutative). Its Jacobson radical $\JJ(R)$ is the intersection of the annihilator ideals of all simple left (equivalently right) $R$-modules, which is also equal to the intersection of all maximal left (equivalently right) ideals. We define a ring to be semisimple if it is artinian and has trivial Jacobson radical. The following theorem is then standard, for instance mostly contained in \cite{Lam}.

\begin{prop}\label{thm:ring}
    The following conditions are equivalent on an artinian ring:
    \begin{enumerate}
    \item  $R$ is absolutely flat, i.e. $\Tor_1^R=0$;
    \item $\Ext^1_R(M,-)=0$, for every finitely presented left $R$-module $M$;
    \item $R$ is a product of simple algebras; 
        \item $R$ is semisimple (i.e. $\JJ(R)=0$);
        \item the only nilpotent two-sided ideal in $R$ is $0$.
\end{enumerate}
If $R$ is commutative, then these are further equivalent to
        \begin{enumerate}
        \item[(1*)] the monoidal category of finitely presented $R$-modules is rigid.
        \end{enumerate}
\end{prop}

\begin{remark}
\begin{enumerate}
\item We could replace condition (2) with the condition that the bifunctor $\Ext^1_R$ is identically zero on the category of left $R$-modules, but for our purposes the current reading is more meaningful, due to Remark~\ref{rem:uExt}. 
    \item Since $R$ is artinian, a module is finitely presented if and only if it is finitely generated if and only if it has finite length.
    \item Absolutely flat rings are more commonly known as Von Neumann regular rings, but their property of being absolutely flat will be most important to us. 
\end{enumerate}
    
\end{remark}

Based on our interest in absolutely flat algebras in tensor categories, the following ring-theoretic observation seems worth pointing out.
\begin{lemma}\label{prop:exam-ring}
Examples of the following types of rings (even $\mC$-algebras) exist:
\begin{enumerate}
    \item a simple (non-artinian) ring that is {\bf not} absolutely flat;
\item a {\bf non}-artinian simple ring that is absolutely flat.
\end{enumerate}
\end{lemma}
\begin{proof}
    
    The Weyl algebra $A$, freely generated over $\mC$ by $x,y$ with $xy-yx=1$ is simple but not absolutely flat. Indeed, we can use the observation that $x\not=xax$ for all $a\in A$, and use \cite[Theorem~4.23]{Lam}.

    Next we define $B$ as the quotient of the endomorphism algebra $B':=\End_{\mC}(V)$, for a countable dimensional vector space $V:= \mC^{\mN}$, by the ideal $J<B'$ that consists of those endomorphisms that have finite-dimensional image. Then $B'$ is absolutely flat by \cite[Proposition~4.27]{Lam}, so $B=B'/J$ is also absolutely flat. Finally, $B$ is not artinian but is simple. 
\end{proof}

\subsection{Tensor categories}

Let $\cC$ be an artinian tensor category over $\bk$ and $A$ an ind-algebra in $\cC$.

\begin{definition}
    The Jacobson radical $\JJ(A)$ of $A$ is the intersection of all annihilator ideals of simple left $A$-modules.
\end{definition}

Before moving on to artinian and semisimple algebras, we discuss some of the alternative definitions that happen to be equivalent in classical ring theory.

\begin{question}\label{q:J}
    Is $\JJ(A)$ equal to the correspondingly defined ideal for right simple modules? 
\end{question}
The answer to Question~\ref{q:J} is affirmative for $A\in\cC$, because then we can establish an annihilator-preserving bijection between simple left and right modules, by applying the monoidal dualities.

\begin{question}
    Under which conditions is the Jacobson radical equal to the intersection of maximal left or right ideals?
\end{question}

We will see in Corollary~\ref{cor:NakSS} that we get this equality in a large number of \emph{commutative} cases. However, it is not always the case:
\begin{example}
    If $\cC$ is not semisimple, then there are always examples of algebras $A\in\cC$ where $\JJ(A)$ is strictly smaller than the intersection of all maximal left ideals in $A$. Moreover, the latter is not necessarily a two-sided ideal. Consider an indecomposable object $X\in\cC$ of length $2$. Then the `matrix algebra' $A:=\uEnd(X)= {^*}X\otimes X$ is simple (as is any internal endomorphism algebra of a compact object), so $\JJ(A)=0$, but $A$ has a unique maximal left ideal ${^*}X\otimes S$, with $S$ the socle of $X$. These claims follow from \cite[Example~7.10.2]{EGNO}.
\end{example}

\begin{remark}
For $\cC$ a finite tensor category and $A\in\cC$, the definition of $\JJ(A)$ agrees with the radical $\mathrm{Rad}^{\cC}(A)$ defined in \cite{CSZ}, as follows for instance from the observation in \cite[Proposition~6.12]{CSZ} that $\mathrm{Rad}^{\cC}(A)$ is the maximal nilpotent ideal and Lemma~\ref{lem:Jart} below.

\end{remark}

\begin{definition}
  The algebra  $A$ is {\bf left} (resp. {\bf right}) {\bf artinian} if for every $X\in\cC$ the $A$-module $A\otimes X$ (resp. $X\otimes A$) is of finite length. $A$ is {\bf artinian} if it is both left and right artinian. 
\end{definition}

\begin{remark}
    If $A$ is left artinian, then the category of finitely presented left $A$-modules equals the category of finitely generated left $A$-modules and the category finite-length left $A$-modules. In particular, it is abelian. The same holds for right artinian algebras and right modules.
\end{remark}

\begin{lemma}\label{lem:Jart}
    If $A$ is artinian (or more generally satisfies the descending chain condition on two-sided ideals), then $\JJ(A)$ is the maximal nilpotent two-sided ideal in $A$.
\end{lemma}
\begin{proof}
    We always have $I<\JJ(A)$ for any nilpotent two-sided ideal $I$, so
    it suffices to show that under our assumptions  $\JJ(A)$ is nilpotent. Indeed, $\JJ(A)^i$ acts trivially on any left $A$-module of length at most $i$. We have $\JJ(A)^{i+1}=\JJ(A)^{i}$ for some $i$. It then follows that $\JJ(A)^i$ acts trivially on all left $A$-modules of finite length, and thus on $A$ itself, implying $\JJ(A)^i=0$. 
\end{proof}

\begin{remark}
    If $\cC$ is finite, then $A$ is left artinian if and only if $A\otimes P$ satisfies the descending chain condition (DCC) on submodules, for every projective $P\in\cC$, as follows from \cite{N}. For general~$\cC$ it does not seem obvious that $A$ must be left artinian if $A\otimes X$ satisfies the descending chain condition for every $X\in\cC$.  However, we show in Section \ref{section artinian comm algs GR MN} that for a commutative algebra $A$ in well-fibered pretannakian categories, $A$ is left artinian if and only if it satisfies DCC on ideals of $A$.
\end{remark}

\begin{remark}
    One could call $A$ {\bf bi-artinian}
    if the $A$-bimodule $A\otimes X\otimes A$ is of finite length for every $X\in\cC$. It can be easily verified that a bi-artinian algebra is artinian. Moreover, in ring theory, so in particular for $\cC=\Vec$, both notions coincide. This is not the case for general artinian tensor categories:

    Consider $\cC=\Rep\mG_a$ the rational representation category of the additive group of the complex numbers. Then the coordinate algebra $A:=\mC[\mG_a]$ is artinian as an algebra in $\Ind\cC$, but not bi-artinian, as the $A$-bimodule $A\otimes A$ has infinite length.
\end{remark}

\begin{definition}
    The algebra $A$ is {\bf semisimple} if it is artinian and $\JJ(A)=0$.
\end{definition}

It follows from Theorem~\ref{prop:consolidate} below that we can replace $\JJ(A)$ in this definition by its right module version.

\begin{lemma}
    The following conditions are equivalent on an ind-algebra $A$ in $\cC$:
    \begin{enumerate}
        \item every left $A$-module is flat;
        \item every right $A$-module is flat.
    \end{enumerate}
    If these conditions are satisfied, we call $A$ {\bf absolutely flat}.
\end{lemma}
\begin{proof}
    There are enough flat modules, namely $A\otimes X$ and $X\otimes A$ for $X\in\cC$, to establish a Tor-theory. It then follows that both conditions are equivalent to $\Tor_1^A\equiv0$, by a standard argument.
\end{proof}

\begin{remark}
 The example $A=\uEnd(X)$, for compact $X\in\cC$ also shows that for semisimple algebras $A$, the $\bk$-algebra $\Gamma(A)=\End(X)$ need not be semisimple (contrary to the case where $\cC$ is symmetric and $A$ commutative). 
\end{remark}

\begin{thm}\label{prop:consolidate}
    For an artinian ind-algebra $A$, consider the properties:
    \begin{enumerate}
    \item $A$ is absolutely flat, i.e. $\Tor_1^A\equiv0$;
    \item $\uExt^1_A(M,-)=0$ for every finitely presented left $A$-module $M$;
    \item $A$ is a product of simple algebras; 
        \item $A$ is semisimple, i.e. $\JJ(A)=0$;
        \item the only nilpotent two-sided ideal in $A$ is $0$;

         \item[(*1*)] the monoidal category of finitely presented $A$-bimodules is rigid.
        
        \end{enumerate}
        We have implications 
        $$({}^*1{}^*)\Leftrightarrow (1)\Rightarrow (2)\quad\mbox{and}\quad (3)\Rightarrow (4)\Leftrightarrow (5).$$
        If the category of finitely presented $A$-bimodules is an abelian subcategory of the category of all $A$-bimodules (e.g. $A$ is bi-artinian or $A\in\cC$), then $(1)\Rightarrow (3)$. 
\end{thm}
\begin{proof}
 The equivalence $(4)\Leftrightarrow (5)$ follows from Lemma~\ref{lem:Jart}.

    It is clear that (3) implies (5). The equivalence of (1) and $^{\ast}(1)^*$ holds even without artinian assumptions, see Proposition~\ref{prop:absflattensor} below.

    Now we prove ${}^\ast(1)^\ast\Rightarrow (2)$. Let $M$ be a finitely presented left $A$-module. Then $M\otimes A$ is a finitely presented $A$-bimodule and we have 
    $$\uHom_A(M,-)\cong\uHom_{A\otimes A}(M\otimes A,-)\;\cong\; N\otimes_A-,$$
when evaluated in $A$-bimodules, where the $A$-bimodule $N$ is the right monoidal dual of $M\otimes A$. The functor $\uHom_A(M,-)$ is thus exact when evaluated on the category of $A$-bimodules. Furthermore, we have
$$\uHom_{A}(M,-\otimes A)\;\cong\; \uHom_A(M,-)\otimes A$$
as functors on the category of left $A$-modules, using the fact that $M$ is finitely presented to bring the ind-object $A$ outside. It follows that $\uHom_{A}(M,-)$ is exact.

    The conditional implication $({}^\ast1^\ast)\Rightarrow (3)$ follows from Lemma~\ref{Lem:DM} below.
\end{proof}

Due to Proposition~\ref{thm:ring} and Theorem~\ref{thm:CSZ} below, we are led to the following question.
\begin{question}\label{mainquestion}
    Under which assumptions are the conditions in Theorem~\ref{prop:consolidate} equivalent?
\end{question}
Note that (4) does not always imply (3), an example will be given below in Lemma~\ref{lem:Dela}.

The following two results are well-known.
\begin{prop}\label{prop:absflattensor}
    The following conditions are equivalent on $A$:
    \begin{enumerate}
        \item $A$ is absolutely flat;
        \item the monoidal category of $A$-bimodules in $\Ind\cC$ has exact tensor product;
        \item finitely presented objects in the monoidal category of $A$-bimodules in $\Ind\cC$ are rigid.
    \end{enumerate}
\end{prop}
\begin{proof}
    That (1) implies (2) is obvious. That (2) implies (3) follows from the standard observation that in an abelian monoidal category with exact tensor product the cokernel of a morphism between two rigid objects is again rigid (with duals given by the kernels of the dual morphisms). Since any object (in a Grothendieck category) is a filtered colimit of finitely presented objects, (3) implies (2). That (2) implies (1) follows by  replacing left $A$-modules $M$ by bimodules $M\otimes A$, the same for right modules, and using exact faithfulness of the functors $-\otimes A$ and $A\otimes-$.
\end{proof}

\begin{lemma}\label{Lem:DM}
    Let $(\cA,\otimes,\unit)$ be an abelian rigid monoidal category with additive tensor product in which $\unit$ has finite length, then $\unit$ is, as an algebra, a product of simple algebras.
\end{lemma}
\begin{proof}
    Let $U\subset\unit$ be a subobject, then $\unit=U\oplus U^\perp$, with $U^\perp=\ker(\unit\to U^\ast)$ by \cite[Proposition~1.17]{DM}. The latter is stated for symmetric tensor categories only, but carries over to the general case. Then it follows that $U^\perp\otimes U=0$, as we can realise it as the kernel of $U\to U^\ast\otimes U$, where the composite with $\ev_U:U^\ast\otimes U\to\unit$ is the inclusion $U\hookrightarrow \unit$. It then follows that $U\otimes U=U$, $U^\perp\otimes U^\perp=U^\perp$ and thus $U\otimes U^\perp=0$. This is thus an algebra factorisation $\unit=U\times U^\perp$. We can continue this procedure on the factors, and the procedure is finite by assumption of finite length of $\unit$.
\end{proof}

\begin{remark}

Proposition~\ref{prop:absflattensor} states that for $A$ absolutely flat with $\Gamma(A)=\bk$, the category of finitely presented $A$-bimodules is an ``ind-tensor category'' over $\bk$, see Appendix~\ref{App}. Assuming it is a tensor category, it is ``$\Hom$-finite'' if and only if
    \begin{equation}\label{eq:finHom}
        \dim_{\bk}\Hom(X,A)\;<\;\infty,\quad\mbox{for all}\quad X\in\cC,
    \end{equation}
    while it is artinian if and only if $A$ is bi-artinian.
\end{remark}

\subsection{The commutative case}
In this subsection, let $\cC$ be a pretannakian category.

\begin{prop}\label{prop:absflattensor2}
    The following conditions are equivalent on a commutative ind-algebra $A$:
    \begin{enumerate}
        \item $A$ is absolutely flat;
        \item the monoidal category $\Mod_A{\cC}$ has exact tensor product;
        \item finitely presented objects in $\Mod_A{\cC}$ are rigid.
    \end{enumerate}
    Consequently, if $A$ is simple and artinian, then $A$ is absolutely flat if and only if $\mod_A{\cC}$ is a pretannakian category over the field $\Gamma(A)$.
\end{prop}
\begin{proof}
    This is proved analogously to Proposition~\ref{prop:absflattensor}, see also \cite[Proposition~6.1.2]{CEO}.
\end{proof}

\begin{remark}
    For an absolutely flat commutative algebra $A$, one can verify that $\mod_A{\cC}$ is precisely the category of rigid objects in $\Mod_A{\cC}$.
\end{remark}

We then have a slightly stronger version of Theorem~\ref{prop:consolidate}.

\begin{thm}
\label{prop:consolidate2}  
For a commutative artinian algebra $A$ consider the properties
    \begin{enumerate}
     \item $A$ is absolutely flat, i.e. $\Tor_1^A=0$;
     \item $\uExt^1_A(M,-)=0$ for every finitely presented $A$-module $M$;
    \item $A$ is a product of simple algebras; 
        \item $A$ is semisimple, i.e. $\JJ(A)=0$;
        \item the only nilpotent ideal in $A$ is $0$;

         \item[(1*)] the monoidal category $\mod_A\cC$ is rigid.

        \end{enumerate}
        We have implications 
        $$(1^*)\Leftrightarrow(1)\Leftrightarrow(2)\Rightarrow(3)\Rightarrow (4)\Leftrightarrow (5).$$
\end{thm}
\begin{proof}
    Other than $(1)\Leftrightarrow(2)$, this follows from Proposition~\ref{prop:absflattensor2}, Lemma~\ref{Lem:DM}, and Theorem~\ref{prop:consolidate}.
    
    We will prove (2)$\Rightarrow$(1*) by showing that when $\uHom_A(M,-)$ is exact for $M\in\mod_A{\cC}$, it follows that $M$ is rigid. By \cite[Proposition~2.3]{Del90}, we need to show that the canonical morphism
    \begin{equation}\label{eq:Hom:Del}\uHom_A(M,A)\otimes_AN\;\to\; \uHom_A(M,N)\end{equation}
    is an isomorphism for every $N\in\mod_A{\cC}$. By taking a presentation of $N$ by free (finitely generated) $A$-modules, and using right exactness of $\otimes_A$ and $\uHom_A(M,-)$, we can reduce the verification of this property to the case $N=A\otimes X$, for $X\in\cC$. By flatness of $A\otimes X$ and right exactness in the contravariant argument of $\uHom_A$, this question now also reduces to the case $M=A\otimes Y$, in which case \eqref{eq:Hom:Del} is an isomorphism, for instance again by \cite[Proposition~2.3]{Del90}.
\end{proof}

   \begin{remark}
       It follows from Proposition~\ref{prop:C} below that $\mod_A\cC$ is rigid if and only if its simple objects are rigid. This also follows from \cite[Corollary~4.7]{EP}.
   \end{remark}

We then have the special case of Question~\ref{mainquestion}.
\begin{question}\label{q:comm}\label{q:main}
    For which pretannakian categories $\cC$ are all properties in Theorem~\ref{prop:consolidate2} equivalent, for all commutative artinian ind-algebras~$A$? Concretely:
    \begin{enumerate}
        \item[(a)] When is every simple commutative artinian ind-algebra in $\cC$ absolutely flat?
        \item[(b)] When is every semisimple commutative ind-algebra in $\cC$ a product of simple algebras?
    \end{enumerate}
\end{question}

Our main results (e.g. Theorem A) indicates that all the properties might indeed be equivalent in pretannakian categories of moderate growth. However, 
as we show in Lemma~\ref{lem:Dela}, at least the property in Question~\ref{q:comm}(b) can fail in pretannakian categories not of moderate growth.
We will also show in Proposition~\ref{lem:counter} that, contrary to the case $\cC=\Vec$, for general $\cC$ we get a negative answer to \ref{q:comm}(a) when we remove the artinian assumption, as in the following question.
\begin{question}\label{q:nonart}
For which pretannakian $\cC$ is every commutative simple ind-algebra
\begin{enumerate}
    \item absolutely flat?
    \item artinian?
\end{enumerate} 
\end{question}
Again, Theorem~A indicates this class should contain all pretannakian categories of moderate growth. Of course the counterexample in Proposition~\ref{lem:counter} is about superexponential growth. 

For our main motivation to study the above questions, the following weaker form is most relevant:
\begin{question}\label{Q4Os}
    Is every finitely generated commutative simple ind-algebra in $\cC$ absolutely flat and artinian?
\end{question}

Recall that a pretannakian category $\cC$ over an algebraically closed field $\bk$ is {\bf incompressible} if every symmetric tensor functor $\cC\to\Ds$ to every pretannakian category $\Ds$ over $\bk$ is an embedding of a tensor subcategory, see \cite{CEO} for details.

\begin{prop} Assume that $\Bbbk$ is an uncountable algebraically closed field of characteristic 0.
If the answer to Question~\ref{Q4Os} is affirmative, then \cite[Conjecture~C]{CEO}, stating that the only incompressible pretannakian categories over $\Bbbk$ are of moderate growth, is true.
\end{prop}
\begin{proof}
Let $\cC$ be pretannakian and $X$ a simple object in $\cC$. The construction in the proof of \cite[Lemme~2.8]{Del02} produces a finitely generated algebra $A$, as a quotient of $S(X\oplus X^\vee)$, such that $A$ is a direct summand of $A\otimes X$ in $\Mod_A\cC$. 

To find a contradiction, we now assume that $\cC$ is not of moderate growth, so that it has a simple object $X\not=\unit$ that is of superexponential growth, $\gd(X)=\infty$. The algebra $A$ is then non-zero, by \cite[Proposition~2.5]{Del02}. Let $B$ be a simple quotient algebra of $A$, so that (since we assume the validity of Question~\ref{Q4Os}) $\mod_B\cC$ is a pretannakian category over $\Gamma(B)$ by Proposition~\ref{prop:absflattensor2}. Since $A$ and $B$ by construction has countable length, it follows that $\Gamma(B)=\Bbbk$ and the non-full symmetric tensor functor
$$B\otimes -:\;\cC\to\Mod_B\cC$$
shows that $\cC$ is not incompressible.
\end{proof}


\section{Examples and counterexamples}\label{sec:examples}

Unless further specified, we let $\bk$ be a field and $\cC$ a tensor category over $\bk$.

\subsection{Finite tensor categories}

The following theorem is the main result of \cite{CSZ} and was conjectured in \cite{EO}. Recall that, following \cite{EGNO}, an algebra $A\in\cC$ is {\bf exact} if for any left $A$-module $M$ in $\cC$, and every projective $P\in\cC$, the $A$-module $M\otimes P$ is projective.
\begin{thm}\label{thm:CSZ}
    If $\cC$ is a finite tensor category and $A\in\cC$, then the conditions in Theorem~\ref{prop:consolidate} are all equivalent, and equivalent to 
    \begin{enumerate}
        \item[(6)] $A$ is exact.
    \end{enumerate}
\end{thm}
\begin{proof}
    It is proved in \cite[Theorem~7.1]{CSZ} that (6) is equivalent with (3) and (5). It is easy to see that (6) is equivalent with (2), see \cite[Theorem~7.10.1]{EGNO} or \cite[Proposition~A.2]{CSZ}. By Theorem~\ref{prop:consolidate} it thus suffices to show that (6) implies (1). By assumption, $-\otimes_A(M\otimes P)$ is exact for every $A$-module $M$  and projective $P\in\Cs$, since every projective $A$-module is a direct summand of a free (and hence flat) module $A\otimes V$, $V\in\cC$. Since $-\otimes P$ is faithful and exact, it follows that $M$ is also flat.
\end{proof}

\subsection{Field extensions}

As a special, deceptively easy looking case of Question~\ref{mainquestion}, we can let $A$ be a field extension $K$ of $\bk$. In this case, the conditions in \ref{prop:consolidate}(3)-(5) are obviously satisfied. Moreover, $K$ is artinian as an algebra in $\Ind\cC\supset\VEC$.

To proceed with this special case we observe that the category $\Mod_{K}\cC$ of $K$-modules in $\Ind\cC$ has a canonical monoidal structure with tensor product $\otimes_K$. Note indeed that left and right $K$-modules in $\Ind\Ds$ are identical as both notions correspond to $\bk$-algebra morphisms $K\to\End(X)$. Moreover, it is clear that the $K$-linear monoidal category $\cC_K:=\mod_K\cC$ satisfies all conditions of a tensor category over $K$ except it is not clear in general that it is rigid, in other words it satisfies (1) and (2) from \S\ref{prel:tc}.

Just as in Theorem~\ref{prop:consolidate2}, one can prove that \ref{prop:consolidate}(1) and (2) are equivalent, and equivalent to $\cC_K$ being a tensor category.
 We refer to \cite[\S 5]{CEOP} for an overview on the question of $\cC_K$ being a tensor category (equivalently, $K$ being absolutely flat). One of the central results stated in \cite{CEOP} is that $\cC_K$ is always a tensor category when $\bk$ is algebraically closed, which was proved in \cite[Lemma~2.2]{EO} for $\cC$ finite.
For completeness we provide a proof for the general case:
\begin{prop}\label{prop:FieldExt}
    If $\bk$ is perfect, then $\cC_K$ is an artinian tensor category for any artinian tensor category $\cC$ and field extension $K:\bk$.
\end{prop}
\begin{proof}
    The condition that $\bk$ is perfect implies that the simple objects in $\Mod_K\cC$ are direct summands of free $K$-modules $K\otimes L$, with $L$ simple in $\cC$ (this can be reduced to the corresponding statement for modules over finite dimensional algebras, see for instance \cite[Lemme~5.2(i)]{Del14}, in which case it follows from separability of finite dimensional simple algebras over $\bk$, see for instance \cite[Theorem~4.5.7]{Fo}). Since these have left and right monoidal duals in $\Mod_K\cC$, it follows from Proposition~\ref{prop:C} and Remark~\ref{rem:C} below that every finite-length object, so every object in $\cC_K$, is rigid. We briefly explain why the assumptions needed to apply Proposition~\ref{prop:C} are satisfied. The category $\cB=\Mod_K\cC$ is abelian, with cocontinuous tensor product $\otimes_K$. Existence of internal homomorphism objects follows from Freyd's special adjoint functor theorem.  Every object in $\Mod_K\cC$ is a quotient of a flat object $K\otimes M$, with $M\in\Ind\cC$. Finally, for every simple object $L_\alpha\in\cC$, consider the injective hull $I^\alpha$ of ${}^\ast L_\alpha$ in $\Ind\cC$ as a direct limit $\varinjlim I_i^\alpha$ in $\cC$.  Take its left dual to produce an inverse system $\{Q_i^\alpha=(I_i^\alpha)^\ast\}$. Then the set of inverse systems $\{K\otimes Q_{\bullet}^\alpha\mid \alpha\}$ in $\Mod_K\cC$ satisfies the final assumption in Appendix~\ref{App:C}.
\end{proof}

It then follows quickly, see for instance \cite[\S 5]{CEOP}, that answering Question~\ref{mainquestion} only requires considering finite purely inseparable extensions $K:\bk$. 

\begin{question}\label{q:fields}
    Let $K:\bk$ be a finite purely inseparable field extension and $\cC$ an artinian tensor category over $\bk$. Is $\cC_K$ a tensor category?
\end{question}

We give an overview of cases for which the answer to this question is known.

\begin{prop}\label{prop:exDel}
    Let $K:\bk$ be a finite purely inseparable field extension and $\cC$ an artinian tensor category over $\bk$. Then the extension of scalars $\cC_K$ of $\cC$ to $K$ is a tensor category over $K$ (i.e. $K$ is absolutely flat in $\cC$) in the following cases:
    \begin{enumerate}
        \item $\cC$ is super-tannakian (i.e. has a symmetric tensor functor to $\sVec_L$ for some field extension $L:\bk$);
        \item $\cC$ is semisimple;
        \item $\cC$ is finite or, more generally, admits a tensor functor to a finite tensor category over~$\Bbbk$.
    \end{enumerate}
\end{prop}

\begin{proof}
    Case (1) is \cite[Theorem~5.3.1]{CEOP}, which generalised \cite[Th\'eor\`eme~5.4]{Del14}. Case (2) is \cite[Theorem~5.4.1]{CEOP}. Case (3) is an example of Theorem~\ref{thm:CSZ}, where the extension via tensor functors follows from \cite[Lemma~5.2.5]{CEOP}.
\end{proof}

\subsection{Superexponential growth}

In this section we show via two examples that the behaviour of commutative algebras in pretannakian categories of superexponential growth is more pathological than our results in moderate growth indicate.

\begin{lemma}\label{lem:Dela}
    For an arbitrary field $\bk$, the Delannoy category $\Ds$ from \cite{HSS} contains a commutative algebra $A\in\Ds$ that has no nilpotent ideals (is semisimple) but is not a product of simple algebras. 
\end{lemma}
\begin{proof} We recall that $\Ds$ is semisimple for any $\bk$, see \cite{HSS}.
By \cite[Theorem~4.7]{HSS} and \cite[4.10, 4.11 and~5.2(b)]{HS}, $\Ds$ contains a simple commutative algebra $B$ (the Schwartz space of the real line), which as an object in $\Ds$ is given by
\[B\;=\; L\oplus \unit\oplus L^\ast,\]
for a simple object $L$. Moreover, since $L\otimes L$ contains only $L$, see \cite[Proposition~7.5]{HSS}, up to direct summands that do not occur in $B$, it follows that $A:=L\oplus \unit$ is a subalgebra with ideal $L$. Since $B$ is simple, the multiplication map $L\otimes L\to L$ cannot be zero, otherwise $L$ would generate a nilpotent ideal in $B$. Hence the ideal $L$ in $A$ is not nilpotent either, so $\mathcal{J}(A)=0$.  However $A$ is not a product of simple algebras.
\end{proof}

To set up the next example, fix $\bk=\mathbb{C}$ and $t\in\mathbb{C}\backslash\mZ$. Then we use the description of Deligne's category $(\Rep GL)_t$ from \cite[\S 10]{Del07} as the universal symmetric tensor category on an object $X_t$ of dimension $t$. We also consider $(\Rep S)_t$ from \cite{Del07}, the universal symmetric tensor category on a special Frobenius algebra $V_t$ of dimension $t$. Both are semisimple pretannakian categories. We have symmetric tensor functors
$$\Phi:(\Rep GL)_t\to(\Rep S)_t,\quad X_t\mapsto V_t$$
and
$$\Psi:(\Rep GL)_t\to(\Rep S)_t\boxtimes \sVec,\quad X_t\mapsto V_t\boxtimes\unit\oplus \unit\boxtimes\unit\oplus \unit\boxtimes \bar{\unit}.$$
Both tensor functors are surjective, and so by semi-simplicity, their right adjoints $\Phi_\ast$ and $\Psi_\ast$ are exact and faithful. It follows that the ind-algebras $A=\Phi_\ast(\unit)$ and $B=\Psi_\ast(\unit)$ in $(\Rep GL)_t$ are such that their module categories are the target categories, see \cite[Lemma~6.2.1]{CEO}. We let $R\in\CAlg(\Rep GL)_t$ denote a simple quotient algebra of $A\otimes B$.

\begin{prop}\label{lem:counter}
    The simple commutative ind-algebra $R$ in $(\Rep GL)_t$ satisfies $\Gamma(R)=\bk$ and is {\bf not} absolutely flat.
\end{prop}
\begin{proof}
We start by arguing that $\Gamma(R)=\bk$. For this we can observe that the lengths of the objects $A$ and $B$ are countable since, by construction, they satisfy \eqref{eq:finHom} and there are only countably many simple objects in $(\Rep GL)_t$. Hence the length of $R$ is also countable, so the field extension $\Gamma(R):\bk$ must be $\bk=\mathbb{C}$ itself.

Consider the category of finitely presented $R$-modules $\cC:=\mod_{R}(\Rep GL)_t$. We have a commutative square of exact $\bk$-linear symmetric monoidal functors, that can all be interpreted as extensions of scalars:
    $$\xymatrix{
(\Rep GL)_t\ar[rr]^-{\Phi=A\otimes-}\ar[d]^-{\Psi=B\otimes-}&& (\Rep S)_t\ar[d]^-{R\otimes_A-}\\
(\Rep S)_t\boxtimes\sVec\ar[rr]^-{R\otimes_B-}&&\cC.
    }$$
Note that the functors with target $\cC$ are exact by absolute flatness of $A,B$, or simply because right exact monoidal functors out of tensor categories must always be exact.

In the proof of \cite[Proposition~5.1.1]{CEO}, it was argued that $\Phi$ and $\Psi$, with $\Phi$ re-interpreted as a functor to $(\Rep S)_t\boxtimes\sVec\supset (\Rep S)_t$ cannot be equalised by a symmetric tensor functor out of $(\Rep S)_t\boxtimes\sVec$. We will use a similar argument to shows that $\cC$ cannot have exact tensor product.

The object $S:=R\otimes X_t$ in $\cC$ has a separable algebra structure, as the image of the separable algebra $V_t$ under $R\otimes_A-$. This means that there exists a bimodule section $S\to S\otimes S$ of the multiplication morphism.

Now assume for a contradiction that $R$ were absolutely flat, so that the tensor product on $\cC:=\mod_{R}(\Rep GL)_t$ is exact.
Using separability of $S$ it then follows that its module category $\Mod_S\cC$, which is defined as the category of $S$-modules in the abelian category $\Ind\cC=\Mod_{R}(\Rep GL)_t$, also has exact tensor product (the main observation being that any $S$-module $M$ is a direct summand of the free $S$-module $S\otimes \omega(M)$, where $\omega$ forgets the $S$-action and retains the underlying object in $\cC$). However, $S$ has a non-zero direct summand with zero second symmetric power. Indeed, consider the image of $ \unit\boxtimes\bar{\unit}$ under $R\otimes_B-$. This summand thus generates a non-zero nilpotent ideal.
    
    We thus have obtained a commutative algebra $S$ in $\Ind\cC$ such that the category of $S$-modules in $\Ind\cC$ has exact tensor product, but which has a non-zero nilpotent ideal. However, for any (non-zero) ideal $I$ in $S$, consider the short exact sequence of $S$-bimodules
    $$0\to I\to S\to S/I\to 0.$$
    Applying $-\otimes_SS/I$ produces an exact sequence, but since the map $S/I\to S/I\otimes_S S/I$ is a monomorphism, we find $I\otimes_SS/I=0$. Applying $I\otimes_S-$ to the same displayed short exact sequence then shows that multiplication induces an isomorphism $I\otimes_SI\to I$, in particular $I$ is idempotent and thus not nilpotent. We hence reached a contradiction, showing that $R$ was not absolutely flat.
\end{proof}

\begin{remark}
    The algebras $A$ and $B$ are not finitely generated (although they are artinian), so we have no control over finiteness conditions on $R$, besides $\Gamma(R)=\bk$. Therefore, we do not know whether $R$ is an example of a non-artinian simple commutative algebra, or an example of a non-absolutely flat artinian simple algebra. 
\end{remark}

\begin{remark}
Even for absolutely flat algebras $A$, the condition in $\Gamma(A)=\bk$ itself does not imply condition \eqref{eq:finHom}. To give a concrete example over $\Bbbk=\mathbb{C}$ take $t\in\mathbb{C}\backslash\mathbb{Z}$ and consider the simple absolutely flat algebras $A_i$ given as $F_\ast(\unit)$ for $F:(\Rep GL)_t\to (\Rep GL)_{t-i}$, $X_t\mapsto \unit^i\oplus X_{t-i}$. Then we have the absolutely flat algebra $A:=\varinjlim_{i}A_i$ satisfying $\Gamma(A)=\varinjlim \Gamma(A_i)=\bk$, while $\Hom(X_t,A)$ is infinite-dimensional.
\end{remark}


\section{Pretannakian categories of moderate growth}\label{sec:modgr1}

In this section, we assume that the base field $\bk$ is algebraically closed and $\Cs,\Ds,\Es$ will always be pretannakian categories over $\bk$.

\subsection{Overview}

It was proved in \cite[Theorem~5.2.1]{CEO} that every pretannakian category of moderate growth admits a tensor functor to an `incompressible' one. Based on results and conjectures in \cite{BEO, C1, GR, CEO} a reasonable hypothetical intrinsic characterisation of these incompressible categories has formed: they are expected to be the pretannakian categories that satisfy the properties GR and MN, recalled below. We henceforth call a pretannakian category $\cC$ that admits a symmetric tensor functor to a pretannakian category $\Ds$ that is GR and MN, \emph{well-fibered}. For example, if we accept \cite[Conjecture~1.4]{BEO} and \cite[Conjecture~3.2.3]{C1} as true, then it would follow that every pretannakian category of moderate growth is well-fibered.

\begin{thm}\label{thm:main} Let $\Cs$ be a well-fibered pretannakian category.
     For a commutative artinian ind-algebra $A$ in $\Cs$, the following properties are all equivalent.
    \begin{enumerate}
     \item $A$ is absolutely flat, i.e. $\Tor_1^A=0$;
     \item $\uExt^1_A(M,-)=0$ for every finitely presented $A$-module $M$;
    \item $A$ is a product of simple algebras; 
        \item $A$ is semisimple, i.e. $\JJ(A)=0$;
        \item the only nilpotent ideal in $A$ is $0$;

         \item[(1*)] the monoidal category $\mod_A\Cs$ is rigid.
        \end{enumerate}
        Moreover, any simple commutative ind-algebra in $\cC$ is artinian.
\end{thm}
\begin{proof}
    Due to Theorem~\ref{prop:consolidate2}, for the equivalent properties, we only need to show that (4) implies (3) and that (3) implies (1). The former will be proved in Corollary~\ref{cor:NakSS} below, and the latter will be Corollary~\ref{cor pretann simple comm alg} below. The latter also contains the final claim of the theorem.
\end{proof}

The advantage of well-fibered categories is that we can exploit the theory of commutative algebra and algebraic geometry as developed in \cite{C1, C2, CS}, which we rely on for the proof of Theorem~\ref{thm:main}.

Using the validity of the last sentence of Theorem~\ref{thm:main}, to be proved in Section~\ref{sec:simpleGRMN}, we also establish the following structural result for artinian algebras. Before we state it, recall that for a commutative algebra $A$, we write $\operatorname{Nil}(A)$ for the intersection of all prime ideals of $A$, or equivalently the union of all nilpotent subobjects of $A$ (\cite[\S2.2.4]{C1}).  We have the following.

\begin{thm}\label{thm artin comm rings GRMN} Let $\Cs$ be a well-fibered pretannakian category.
    For $A$ a commutative ind-algebra in $\Cs$, the following are equivalent:
    \begin{enumerate}
        \item $A$ is artinian;
        \item $A$ satisfies the descending chain condition (DCC) on ideals;
        \item $A$ satisfies the ascending chain condition (ACC) on ideals, and every prime ideal is maximal;
        \item $A$ satisfies ACC on ideals, and $A/\operatorname{Nil}A$ is semisimple.
        \item $A$ is a finite direct product of artinian local algebras.
    \end{enumerate}
\end{thm}

This follows from Theorem~\ref{thm:artinianv2} and Corollary~\ref{cor pretann simple comm alg} below, and we stress that our arguments in  Section~\ref{sec:simpleGRMN} do not rely on this Theorem~\ref{thm artin comm rings GRMN}

\subsection{Properties GR and MN}We recall from \cite{C1} the following definitions.

\begin{definition}
    We say $\Ds$ is 
    \begin{enumerate}
        \item geometrically reductive (GR) if for every epimorphism $X\onto\mathbf{1}$ in $\Ds$ there exists $n>0$ such that $S^nX\onto\mathbf{1}$ is split;
        \item maximally nilpotent (MN) if it satisfies 
        \begin{enumerate}
            \item[(i)] if $L$ is a simple, nontrivial object then the symmetric algebra $SL$ is finite;
            \item[(ii)] if $\mathbf{1}\hookrightarrow X$ is non-split, then there exists $n\in\N$ such that $\mathbf{1}\to S^nX$ is 0.
        \end{enumerate}
    \end{enumerate}
\end{definition}

In \cite{C1} and \cite{C2}, it is shown that the conditions GR and MN are suitable conditions for a theory of algebraic geometry to be developed, and this was further justified in \cite{CS}. On the other hand, these properties are expected to be equivalent to the incompressibility of $\Ds$ (for more on connections between these two properties see \cite{C1} and \cite{CEO} and see \cite{GR} for a proof that $\Ver_{p^\infty}$ is GR).  Note that by the following lemma, such categories are always of moderate growth.

\begin{lemma}
    If $\Ds$ satisfies MN(i), then it is of moderate growth.
\end{lemma}

\begin{proof}
If $\Bbbk$ is of characteristic 0, then this follows from \cite[Prop.~0.5]{Del02}.  Now suppose that $\Bbbk$ is of positive characteristic.  We refer to \cite[Defn.~4.1]{CEO-A} for the definition of the quantities $\mathsf{gd}(X)$ (see also \eqref{def:gd}), $\mathsf{sd}(X)$, and $\mathsf{ad}(X)$ in $\R_{\ge 0}\cup\{\infty\}$, for any object $X\in\Ds$.  By our assumption on characteristic, we have the following properties (see \cite[Prop.~4.7, Lem.~4.9]{CEO-A}):
\begin{enumerate}
    \item $\mathsf{gd}(X)\leq \ad(X)$,
    \item $\mathsf{sd}(X)<\infty\Rightarrow\ad(X)<\infty$,
    \item $\ad(X_1\oplus X_2)=\ad(X_1)+\ad(X_2)$, and
    \item $\mathsf{gd}(X)=\mathsf{gd}(\gr X)$ with respect to any finite filtration on $X$.
\end{enumerate}
    By \cite[Cor.~4.11]{CEO-A}, $\Ds$ is of moderate growth if and only if for all $X\in\Ds$ we have $\mathsf{gd}(X)<\infty$.   Taking the Jordan-Holder filtration and applying (4), we may assume that $X$ is semisimple.  Applying (1) and (3), it suffices to show that $\ad(L)<\infty$ for all simple objects $L$, and then applying (2) reduces it to showing $\mathsf{sd}(L)<\infty$ for all simple objects $L$.  If $L=\mathbf{1}$ this is well known, and for $L$ nontrivial we may apply our assumption MN(i).
\end{proof}

\subsection{The Nakayama property and its consequences}

\subsubsection{Nakayama property}\label{section nakayama} We recall from \cite[\S 6.1]{C1} the Nakayama property of a pretannakian tensor category $\mathscr{E}$, which states that for every commutative algebra $A$ of $\mathscr{E}$, ideal $I\leq A$, and finitely generated $A$-module $N$ such that $IN=N$, we have
\[
\operatorname{Ann}_A(N)+I=A.
\]
By \cite[Thm.~6.1.5]{C1}, $\Cs$ as in Theorem~\ref{thm:main} satisfies the Nakayama property.

\medskip

\emph{Now we let $\Es$ be any pretannakian category satisfying the Nakayama property.} 

\medskip

\begin{lemma}\label{lemma simples GR MN}
    Let $A$ be a commutative ind-algebra in $\Es$.
    \begin{enumerate}
        \item If $S$ is a simple $A$-module, then $\operatorname{Ann}_A(S)$ is a maximal ideal.
        \item If $I$ is a minimal non-zero ideal of $A$ such that $I^2=I$, then $A\cong A/I\times A/\operatorname{Ann}_A(I)$.
    \end{enumerate}
\end{lemma}

\begin{proof}
    For (1), suppose that $\operatorname{Ann}_A(S)$ is not maximal, and choose $M$ a maximal ideal which contains it.  Then necessarily $M S\neq 0$, so by simplicity $M S=S$.  Since $S$ is simple, it is clearly finitely generated.  Thus by the Nakayama property, we must have $M=M+\operatorname{Ann}_A(S)=A$, a contradiction.  

    For (2), observe that $I$ must be a simple $A$-module, so by (1) its annihilator is a maximal ideal $M$.  Clearly $M\cap I=0$ and $I+M=A$, so we may apply Lemma \ref{lemma CRT} to obtain the desired isomorphism.
\end{proof}

Lemma \ref{lemma simples GR MN} (1) implies that maximal ideals and annihilator ideals of simple modules are the same, and hence:
\begin{cor}\label{cor jacobson rad GR MN}
    If $A$ is a commutative ind-algebra in $\Es$, then $\mathcal{J}(A)$ is equal to the intersection of all maximal ideals of $A$.
\end{cor}

\begin{remark}
    Observe that Lemma \ref{lem:Dela} shows that Lemma \ref{lemma simples GR MN} need not hold in an arbitrary pretannakian category.  In particular we see that the Nakayama property fails to hold in the Delannoy category. 
\end{remark}

\begin{cor}\label{cor semisimple prod of simple algs}
    Let $A$ be a semisimple commutative ind-algebra in  $\Es$.  Then $A$ is a product of simple algebras.
\end{cor}

\begin{proof}
   Let $I$ be a minimal ideal of $A$,  which exists because semisimple algebras are assumed to be artinian. Then $I$ is a simple $A$-module, and since $A$ is semisimple, we must have $I^2=I$.  Thus by Lemma \ref{lemma simples GR MN}, $\operatorname{Ann}_A(I)$ is a maximal ideal and $A\cong A/I\times A/\operatorname{Ann}_A(I)$.  Inducting on the length of $A$ as a module over itself then gives the result.     
\end{proof}

We now specialise the above results to our main case of interest.
\begin{cor}\label{cor:NakSS}
    Let $\cC$ be a well-fibered pretannakian category and let $A$ be a commutative ind-algebra in  $\Cs$.
    \begin{enumerate}
        \item If $A$ is semisimple, then $A$ is a product of simple algebras.
        \item $\mathcal{J}(A)$ is equal to the intersection of all maximal ideals of $A$.
    \end{enumerate}
\end{cor}

\medskip

\subsection{Artinian commutative algebras}\label{section artinian comm algs GR MN}

\begin{thm}\label{thm:artinianv2} Let $\Es$ be a pretannakian category satisfying the Nakayama property and in which every simple commutative ind-algebra is artinian. 
    For $A$ a commutative ind-algebra in $\Es$, the following are equivalent:
    \begin{enumerate}
        \item $A$ is artinian;
        \item $A$ satisfies the descending chain condition (DCC) on ideals;
        \item $A$ satisfies the ascending chain condition (ACC) on ideals, and every prime ideal is maximal;
        \item $A$ satisfies ACC on ideals, and $A/\operatorname{Nil}A$ is semisimple.
        \item $A$ is a finite direct product of artinian local algebras.
    \end{enumerate}
\end{thm}

We now fix $\Es$ as in the theorem, and start the proof.

\begin{lemma}\label{cor finite lenght implies artinian}
    If $A$ is a commutative ind-algebra in $\Es$ which is of finite length as a module over itself, then it is artinian.
\end{lemma}

\begin{proof}
    Indeed, let $JH^\bullet(A)$ be a finite composition series  for $A$, as a module over itself. For an object $X\in\Cs$, the $A$-module $(JH^i(A)/JH^{i+1}(A))\otimes X$, which factors through a simple quotient of $A$ by Lemma~\ref{lemma simples GR MN}(1), must be of finite length by our assumption that simple commutative ind-algebras in $\Es$ are artinian. Hence also $A\otimes X$ is of finite length.
\end{proof}

\begin{remark}
    A commutative ind-algebra $A$ in $\Cs$ is called Noetherian if every submodule of $A\otimes X$ is finitely generated for all $X\in\Cs$ (\cite[\S 6.4]{C1}). We will show that DCC on ideals is equivalent to being artinian for commutative algebras in $\Cs$, see Theorem \ref{thm artin comm rings GRMN}, so it is natural to hope that in $\Cs$ we also have that ACC on ideals implies being Noetherian, but we do not know if this holds.
\end{remark}

\begin{lemma}\label{lemma DCC}
    Let $A$ be a commutative ind-algebra in $\Es$ satisfying DCC on ideals. Then the following hold.
    \begin{enumerate}
        \item Every prime ideal of $A$ is maximal.
        \item $\JJ(A)=\operatorname{Nil}(A).$ 
        \item $\operatorname{Nil}(A)$ is a nilpotent ideal.
        \item $A$ has only finitely many maximal ideals.
    \end{enumerate}
\end{lemma}

\begin{proof}
    Let us start with (1).  By taking a quotient, it suffices to prove that if $A$ is an integral domain (i.e.~$(0)$ is a prime ideal of $A$), then $A$ is simple.  Suppose otherwise.  Then since $A$ satisfies DCC on ideals, it contains a minimal non-zero proper ideal $I$.  Since $A$ is an integral domain, we must have $I^2=I$, so by Lemma \ref{lemma simples GR MN}, $A\cong A/I\times A/\operatorname{Ann}_A(I)$, which contradicts the fact that $A$ is an integral domain.  It follows that $A$ must be simple.

    For (2), we may apply (1) to obtain that $\operatorname{Nil}(A)$ is the intersection of all maximal ideals of~$A$.  The conclusion follows from Corollary~\ref{cor jacobson rad GR MN}.

    For (3) and (4), the proofs in \cite[Prop.~8.3 and Prop.~8.4]{AM} carry over almost verbatim.
\end{proof}

\begin{lemma}\label{lemma local artin alg}
    Let $A$ be a local ind-algebra in $\Es$ satisfying ACC on ideals, such that its maximal ideal is the unique prime ideal.  Then $A$ is artinian.
\end{lemma}

\begin{proof}
    Let $M$ denote the maximal ideal of $A$.  Then we claim that $M^k=0$ for some $k$.  Let $\mathcal{S}$ denote the set of ideals $I$ of $A$ such that $M^k\not\subseteq I$ for any $k\in\N$, and suppose it is nonempty.  Then because $A$ satisfies ACC on ideals, there exists a maximal element $J$ of $\mathcal{S}$.  Since $J\neq M$, the ideal $J$ is not prime, and thus there exists ideals $I_1,I_2$ not in $J$ such that $I_1I_2\subseteq J$.  But then $J+I_1,J+I_2\notin\mathcal{S}$, so there exists $k\in\N$ such that $M^k$ lies in $J+I_1$ and $J+I_2$.  Hence $M^{2k}\sub(J+I_1)(J+I_2)=J^2+JI_1+JI_2+I_1I_2\sub J$, a contradiction.  It follows that $\mathcal{S}$ must be empty, and hence there exists $k\in\N$ such that $M^k=0$.

    Now we have the chain of ideals
    \[
    0=M^k\sub M^{k-1}\sub\cdots\sub M \sub A,
    \]
    and each $M^i/M^{i+1}$ is a finitely generated $A/M$-module because $A$ satisfies ACC on ideals.  Since $A/M$ is simple, it follows by assumption that $A/M$ is artinian, and thus $M^i/M^{i+1}$ has finite length as an $A$-module.  We can now deduce that $A$ itself is of finite length as a module over itself.  
    We may now conclude by Lemma \ref{cor finite lenght implies artinian}.    
\end{proof}

\begin{proof}[Proof of Theorem~\ref{thm:artinianv2}]
    It is clear that $(1)\Rightarrow(2)$ and $(5)\Rightarrow(1)$. Assume (2) holds; we will now prove that both (3) and (5) hold.  Lemma \ref{lemma DCC} implies every prime ideal is maximal, and there are only finitely many maximal ideals.  Thus let $M_1,\dots,M_n$ denote the distinct maximal ideals of $A$.  By Corollary~\ref{cor jacobson rad GR MN} and Lemma \ref{lemma DCC} (2), we have $M_1\cdots M_n\sub\operatorname{Nil}(A)$, and $(M_1\cdots M_n)^k=0$ for some $k>0$.  By the obvious generalisation of \cite[Cor.~6.11]{AM}, we obtain that $A$ satisfies ACC on ideals, so (3) holds.  Moreover, since $M_i+M_j=A$ for $i\neq j$, $M_i^k+M_j^k=A$, so Lemma \ref{lemma CRT} implies that
    \[
    A\cong A/M_1^k\times \cdots \times A/M_n^k.
    \]
    It is easy to see that $A/M_i^k$ is a local algebra which satisfies ACC on ideals, and that its maximal ideal is the unique prime ideal.  Thus by Lemma \ref{lemma local artin alg}, $A/M_i^k$ is artinian for all $i$, and therefore (5) holds.  
    
    We now prove $(3)\Rightarrow(4)$.  By our assumption and Corollary \ref{cor jacobson rad GR MN}, we know that $\operatorname{Nil}(A)$ is the intersection of all maximal ideals $M$ of $A$.  Since $A$ satisfies ACC on ideals, $\Spec(A)$ (\cite[\S2.2]{C1}) must be a Noetherian topological space, and therefore is a union of finitely many irreducible components, each of which is necessarily a point.  Thus $A$ has finitely many maximal ideals $M_1,\dots,M_n$, and $\operatorname{Nil}(A)=M_1\cap\cdots \cap M_n$.  By Lemma \ref{lemma CRT} we obtain that
    \[
    A/\operatorname{Nil}(A)\cong A/M_1\times\cdots\times A/M_n,
    \]
    which is semisimple, hence (4) holds.

    
     Finally, we prove $(4)\Rightarrow(2)$.  Becasue $A$ satisfies ACC on ideals, every ideal of $A$ is compactly generated.  In particular, $\operatorname{Nil}(A)$ is generated by a compact object $X\in\operatorname{Nil}(A)$.  Since $X$ is itself nilpotent, this implies $\operatorname{Nil}(A)$ is a nilpotent ideal.  Thus powers of $\operatorname{Nil}(A)$ provide a finite filtration on $A$.  Since each $\operatorname{Nil}(A)^i$ is compactly generated, the layers of this filtration will be finitely generated modules over $A/\operatorname{Nil}(A)$, which is semisimple, and thus artinian by assumption.  It follows that $A$ has finite length as a module over itself, and therefore must satisfy the DCC condition on ideals, proving (2). 
    \end{proof}

\begin{proof}[Proof of Theorem~\ref{thm artin comm rings GRMN}]
This follows immediately from the combination of Theorem~\ref{thm:artinianv2} with the last line of Theorem~\ref{thm:main} and the fact that $\cC$ satisfies the Nakayama property, see \ref{section nakayama}.
\end{proof}

\subsection{The noncommutative case for finite algebras}

In this section we prove the following.

\begin{thm}\label{thm noncomm case}
    Let $\Cs$ be a well-fibered category in characteristic $p>0$, or a tannakian category in characteristic 0.  If $A\in\Cs$, then the conditions of Theorem \ref{prop:consolidate} are all equivalent.
\end{thm}

Theorem \ref{thm noncomm case} covers all well-fibered categories apart from one class, which we leave as an open question.

\begin{question}\label{question:super}
    If $\Cs$ is a supertannakian category in characteristic 0 and $A\in\Cs$, are the conditions of Theorem \ref{prop:consolidate} all equivalent?
\end{question}

We now assume that we have a tensor functor $\Cs\to\Ds$, where $\Ds$ is a GR+MN pretannakian category.  Since $A\in\Cs$, we may assume that $\Cs$ is finitely generated.  By tannakian formalism (see for instance \cite[\S 4]{CEO}), there exists an algebraic group $G$ in $\Ds$ such that $\Cs$ embeds in $\Rep^f G$, where we write $\Rep^f G$ for the tensor category of finite-length $G$-modules in $\Ds$.  Therefore it suffices to prove Theorem \ref{thm noncomm case} for $\Cs=\Rep^f G$, and we assume from now on that $G$ is an algebraic group in $\Ds$ and $A$ is an algebra in $\Rep^f G$ (in particular, it is of finite length).

For such an algebraic group $G$, recall that we have a subgroup $G_{\mathrm{red}}\leq G$, where $\Bbbk[G_{\mathrm{red}}]=\Bbbk[G]/\text{Nil}$, where $\text{Nil}$ is the nilradical of $\Bbbk[G]$.  It is straightforward to check the following:
\begin{enumerate}
    \item $G_{\mathrm{red}}$ is an algebraic group in $\Vec$;
    \item $G_{\mathrm{red}}(\Bbbk)=G(\Bbbk)$; and,
    \item the natural functor $\Rep^f G_{\mathrm{red}}\to\mod_{\Bbbk G(\Bbbk)}\Ds$ is fully faithful.
\end{enumerate}

Next, recall from \cite[\S 4]{CS} that if $\Bbbk$ has characteristic $p>0$, then for $r\in\N$ one may define the $r$th Frobenius kernel of $G$, denoted $G_r$. Then $G_r$ is a normal, finite subgroup scheme in $G$, and we have inclusions $G_{r}\leq G_{r+1}$ for all $r$.  Here, we say a group scheme $H$ is finite if $\Rep^f H$ is a finite tensor category, or equivalently $\Bbbk[H]$ is of finite length, meaning $\bk[H]\in\Ds$.

\begin{lemma}\label{lemma finite normal sub}
    Suppose that $\Ds=\Vec$ or $\Bbbk$ is of positive characteristic.  Then there exists a normal, finite subgroup scheme $N\trianglelefteq G$ such that $G=NG_{\mathrm{red}}$.
\end{lemma}

\begin{proof}
    If $\Bbbk$ is of characteristic 0 (so that $\Ds=\Vec$), then $G=G_{\mathrm{red}}$ and we may take $N=\{e\}$.   If $\Bbbk$ is characteristic $p>0$, then by \cite[Cor.~4.24]{CS} we have $G=G_rG_{\mathrm{red}}$ for $r$ large enough.
\end{proof}

We now proceed to the proof of Theorem \ref{thm noncomm case}.  For this, we prove the following slightly stronger statement.  

\begin{thm}\label{thm reduced finite}
Suppose that there exists a normal finite subgroup scheme $N\trianglelefteq G$ such that $G=NG_{\mathrm{red}}$.  Then for an algebra $A\in\Rep^f G$, the conditions of Theorem \ref{prop:consolidate} are all equivalent.
\end{thm}

By Lemma \ref{lemma finite normal sub}, the hypothesis of Theorem \ref{thm reduced finite} applies whenever $\Ds$ is of positive characteristic or $\Ds=\Vec$, so that Theorem~\ref{thm noncomm case} follows.  We note that if $\Ds=\sVec$ and $\Bbbk$ is characteristic 0, then the hypothesis of Theorem \ref{thm reduced finite} holds if and only if $(\operatorname{Lie}G)_1$ is an ideal, 
giving only a partial answer to Question~\ref{question:super}.

\subsubsection{Toward the proof of Theorem \ref{thm reduced finite}}

Fix a normal subgroup $N\trianglelefteq G$, and write $R:\Rep^f G\to\Rep^f N$ for the restriction functor.  Note we do not yet assume $N$ to be finite.  Observe that $R(A)$ is an algebra in $\Rep^f N$.  By abuse of notation, we will also write $R$ for the natural functor $\mod_{A}\Rep^f G\to\mod_{R(A)}\Rep^f N$.  



\begin{lemma}\label{lemma res semisimple}
    Suppose that $N$ is a normal subgroup scheme of $G$ such that $G=NG_{\mathrm{red}}$.  \begin{enumerate}
        \item If $L\in\mod_A\Rep^f G$ is simple, then $R(L)$ is a semisimple $R(A)$-module.
        \item If $L'\in\mod_{R(A)}\Rep^f N$ is simple, then there exists a simple $A$-module $L$ such that $L'$ is a direct summand of $R(L)$ in $\mod_{R(A)}\Rep^f N$.
    \end{enumerate}
\end{lemma}

\begin{proof}
    For (1), choose a simple $R(A)$-submodule $L'\sub R(L)$.  Then for $g\in G(\Bbbk)$, $g(L')$ will be another simple $R(A)$-submodule of $L$.  Moreover, $V=\sum\limits_{g\in G(\Bbbk)}g(L')$ is a $G_{\mathrm{red}}$-stable, $R(A)$-submodule of $L$.  Since $G=NG_{\mathrm{red}}$, $V$ is thus $G$-stable, and hence defines a non-zero $A$-submodule of $L$.  By simplicity of $L$ this forces $V=L$, and since $V$ is clearly semisimple in $\mod_{R(A)}\Rep^f N$ we are done.

     To prove (2), we may apply (1) so that it suffices to show that $L'$ is a quotient of $R(L)$ for some simple $A$-module, and moreover for this it only suffices to show it is a quotient of some (not necessarily simple) $A$-module.  Notice that the quotient of schemes $G/N$ exists and is affine by \cite[Cor.~7.10]{CS}, and thus by \cite[Thm.~7.3.1]{C2} there exists $Y\in\Rep^fG$ such that $(L')^*\sub R(Y)$.  Setting $X=A\otimes Y^*$, we obtain a natural surjective map of $R(A)$-modules $R(A\otimes Y)=A^{R}\otimes R( Y)\onto L'$ and we are done.
\end{proof}

Lemma \ref{lemma res semisimple} has the following corollary which is of independent interest.
\begin{cor}
    Continuing with our hypotheses on $G$ and $N$, the restriction functor $\operatorname{Res}_N^G:\Rep^f G\to\Rep^f N$ takes simple modules to semisimple modules.  Moreover, every simple in $\Rep^f N$ is a direct summand of $\operatorname{Res}_N^G(L)$ for some simple $L\in\Rep^f G$.
\end{cor}

\begin{proof}
    Indeed, this follows immediately from Lemma \ref{lemma res semisimple} when $A=\mathbf{1}$.
\end{proof}

\begin{proof}[Proof of Theorem \ref{thm reduced finite}]
    Due to Theorem \ref{prop:consolidate} itself, it suffices to prove the implications $(4)\Rightarrow(1)$ and $(2)\Rightarrow (4)$. Before we begin, we claim that $R(\mathcal{J}(A))=\mathcal{J}(R(A))$.  Indeed, observe that if $I\leq R(A)$ is a nilpotent ideal, then so is $g(I)$ for any $g\in G(\Bbbk)$, and from this the result easily follows.
    
    Now assume (4) holds so that $A$ is semisimple, which implies $R(A)$ is also semisimple.  By Theorem \ref{thm:CSZ}, it follows that $R(A)$ is absolutely flat.  Since we have $R(M\otimes_A N)\cong R(M)\otimes_{R(A)}R(N)$ and $R$ reflects exact sequences, it follows that $A$ is also absolutely flat, proving (1).  
    

    Now assume (2) holds.  If $A$ is not semisimple, then neither is $R(A)$.  Hence by Theorem~\ref{thm:CSZ}, there exists simple $L_1',L_2'\in\mod_{R(A)}\Rep^f N$ such that $\uExt^1_{R(A)}(L_1',L_2')\neq0$.  By Lemma \ref{lemma res semisimple}, there exists simple $L_1,L_2\in\mod_A\Rep^f G$ such that $L_i'$ is a direct summand of $R(L_i)$ for $i=1,2$.  It follows that $R(\uExt^1_A(L_1,L_2))=\uExt^1_{R(A)}(R(L_1),R(L_2))$ contains $\uExt^1_{R(A)}(L_1',L_2')$ as a direct summand, and is therefore non-zero.  This contradicts our assumption, meaning that $A$ must instead be semisimple, and we are done.
\end{proof}


\section{Simple commutative ind-algebras in moderate growth pretannakian categories}\label{sec:simpleGRMN}
We assume for this section that $\Ds$ is a pretannakian category satisfying GR and MN.

\medskip

Let $G$ be an affine group scheme in $\Ds$, see for example \cite[\S 7]{C1}, \cite{BP} or \cite[\S 3]{CS}. We write $\Rep G$ for the monoidal category of $\bk[G]$-comodules in $\Ind\Ds$, which is the ind-completion of the tensor category $\mathrm{Rep}^f G$ of $G$-representations in $\Ds$.  For a subgroup $H\subseteq G$ we write $\operatorname{Ind}_{H}^{G}:\Rep(H)\to\Rep(G)$ for the right adjoint to the monoidal restriction functor $\operatorname{Res}^G_H:\Rep(G)\to\Rep(H)$.  Then $\operatorname{Ind}_H^{G}(\mathbf{1})\cong \bk[G]^H$ is the subalgebra of $H$-invariants in $\bk[G]$ by \cite[Lem.~7.13]{CS}, and the subalgebra structure agrees with the algebra coming from the lax monoidal functor $\Ind^G_H$ acting on the algebra $\unit$. The subgroup $H$ is called {\bf exact} if $\operatorname{Ind}_H^{G}$ is an exact functor.  

\begin{lemma}\label{lemma induced from exact is simple}
    If $H\sub G$ is exact, then $\Bbbk[G]^H$ is a simple commutative algebra in $\Rep G$, and we have an equivalence of symmetric monoidal categories $\Rep H\cong\Mod_{\Bbbk[G]^H}(\mathrm{Rep}^fG)$.  
\end{lemma}

\begin{proof}The proof of the case $\Ds=\Vec$, see \cite[Theorem~6.1.5(1)]{CEO}, carries over:
    By Barr-Beck monadicity, or \cite[Prop.~7.12]{CS} and \cite[Thm.~7.3.2]{C2} when $G$ is of finite type, we have $\Rep H\cong\Mod_{\Bbbk[G]^H}$ as symmetric monoidal categories.  This implies that $\Bbbk[G]^H$, which is the unit in $\mod_{\Bbbk[G]^H}$, is simple as a module over itself in $\Rep(G)$, which is equivalent to simplicity as an algebra. 
\end{proof}
Lemma \ref{lemma induced from exact is simple} then implies, for instance by Theorem~\ref{prop:consolidate2}:
\begin{cor}\label{cor exact is flat}
    If $H\sub G$ is exact, then $\Bbbk[G]^H$ is absolutely flat in $\Rep G$.
\end{cor}

We will reserve the notation $\Gamma=\Hom(\unit,-)$ for the functor $\Ind\Ds\to\VEC$, even when applied to objects in $\Rep G$. We will not need the $\Gamma$-notation for $\Rep G$ directly due to the following lemma.
\begin{lemma}\label{lem:AG} Let $A$ be a simple commutative algebra in $\Rep G$.
With $\Ind\Ds\ni A^G\sub A$ the subalgebra of $G$-invariants,  we have equalities
    \[\Hom_{\Rep G}(\unit,A)=\Hom_{\Ds}(\unit,A^G)= A^G.\]
    So $A^G\in\VEC$ and $A^G$ is a field.
\end{lemma}
\begin{proof} The left equality is by definition.
    For the second, it suffices to show that $A^G$ is a simple algebra, so that it must be in $\VEC\subset\Ind\Ds$. For a contradiction consider an ideal $I$ in $A^G$. Then $I$ does not belong to $\VEC$ since $\Gamma(A^G)$ is simple. Hence at least one of the following is true: either there exists a non-split $\unit\hookrightarrow I$ or a simple subobject $\unit\not=L\subset I$. By the (MN) hypothesis, this implies that in either case, we have a nilpotent subobject $X\subset A^G$ in $\Ds$ which thus generates a proper $G$-invariant ideal $AX$ in $A$, a contradiction.
\end{proof}

The main aim of this section is to prove the following theorem. For a field extension $\mathbb{L}:\bk$, we denote by $G_{\mathbb{L}}$ the affine group scheme in $\Ds_\mathbb{L}$, represented by $\mathbb{L}\otimes_{\bk}\bk[G]$.

\begin{thm}\label{thm simple comm alg}
    Let $A$ be a simple commutative algebra in $\Rep G$, and let $\mathbb{L}$ be the algebraic closure of the field $A^G$.   Then there exists an exact subgroup $H_{\mathbb{L}}\sub G_{\mathbb{L}}$ such that $A\otimes_{A^G}\mathbb{L}\cong \mathbb{L}[G_{\mathbb{L}}]^{H_{\mathbb{L}}}$. 
\end{thm}

A simple corollary which seems difficult to show directly is the following.

\begin{cor}\label{cor:kpoint}
    If $A$ is a simple commutative algebra in $\Rep G$ with  $A^G=\Bbbk$, then $A$ has a $\Bbbk$-point.
\end{cor}

A more significant corollary is the following. 
\begin{cor}\label{cor pretann simple comm alg}
If $\Cs$ is a pretannakian category over $\bk$ and admits a symmetric tensor functor to $\Ds$, then every simple commutative ind-algebra $A$ in $\Cs$ is absolutely flat and artinian.  In particular, conditions (1), (2), and (3) in Theorem~\ref{prop:consolidate2} are equivalent for commutative artinian ind-algebras in $\Cs$.
\end{cor}

\begin{proof}
    By tannakian formalism, $\Cs\cong\mathrm{Rep}^f(G,\phi)$ for some constraint $\phi:\pi(\Ds)\to G$, see \cite[\S 8]{Del90}.  This implies that $\Cs$ is a (full) tensor subcategory of $\mathrm{Rep}^f G$, and thus any simple commutative ind-algebra $A$ in $\Cs$ is a simple commutative algebra in $\Rep G$.  Applying Theorem~\ref{thm simple comm alg} and Corollary \ref{cor exact is flat}, we obtain that $A\otimes_{A^G}\mathbb{L}$ is absolutely flat and artinian in $\Cs_{\mL}\subset \mathrm{Rep}^f G_{\mL}$. Now consider $M\in \Mod_A\Cs $. To establish exactness of $M\otimes_A-$, it suffices to show that
    $$\mL\otimes_{A^G} M\otimes_A-\;\cong\; \mL\otimes_{A^G} M\otimes_{A\otimes_{A^G}\mL}(-\otimes_{A^G}\mL)$$
    is exact (as a functor from $\Mod_A{\Cs}$ to $\Ind\Cs$ or to $\Mod_{A\otimes_{A^G}\mL}\Cs$). This now follows from the absolute flatness of $A\otimes_{A^G}\mL$. That $A$ is artinian is also inherited from $A\otimes_{A^G}\mL$.
\end{proof}

Note that Corollary \ref{cor pretann simple comm alg} is stronger than the hoped for implication (3)$\Rightarrow$(2) in Theorem~\ref{prop:consolidate2} because we do not assume that $A$ is artinian.

In the following subsections we start the preparation for the proof of Theorem~\ref{thm simple comm alg}.

\subsection{Some commutative algebra}

Recall $Q(-)$ from \S\ref{sec:comalg1}.
\begin{lemma}\label{lemma open iso field frac}
    Suppose that $X$ is a finite-type affine scheme in $\Ds$, and $i:U\hookrightarrow X$ is an open immersion such that $i^*:\Bbbk[X]\hookrightarrow\Bbbk[U]$ is a monomorphism.  Further suppose that $f\in\Bbbk[U]$ is a zero divisor if and only if its restriction to some connected component of $U$ is nilpotent. Then $i^*$ induces an isomorphism $Q(\Bbbk[X])\to Q(\Bbbk[U])$.
\end{lemma}

\begin{proof}
    Observe that if $f\in\Bbbk[X]$ is a non-zero divisor then it must remain so in $\Bbbk[D(g)]=\Bbbk[X]_g$ for all non-nilpotent $g\in\Gamma(\Bbbk[X])$. Since we may cover $U$ by distinguished open subsets $D(g)$, by the fundamentals of algebraic geometry in \cite{C2}, it follows that $i^*(f)$ remains a non-zero divisor.  

    Thus we obtain an algebra morphism $i^*:Q(\Bbbk[X])\to Q(\Bbbk[U])$.  To see that it is an isomorphism, it suffices to find an $h\in \Bbbk[X]$ which does not become a zero divisor on $U$ and with $D(h)\sub U$.  Indeed, in this case we have an inclusion $\Bbbk[U]\sub\Bbbk[D(h)]$ by \cite[Exer.~II.2.16]{H}. The inclusions $\bk[X]\sub \bk[U]\sub \bk[D(h)]$ then induce canonical inclusions $Q(\bk[X])\sub Q(\bk[U])\sub Q(\bk[X])$.  To find $h$, write $U=U_1\sqcup\cdots\sqcup U_r$ for the connected components of $U$.  Let $h_i\in\Bbbk[X]$ be such that $D(h_i)\sub U_i$ is nonempty.  Then $h=h_1+\dots+h_r$ has the desired property. 
\end{proof}

\subsection{The dual Hopf algebra}
In this subsection, {\em we will always assume that $G$ is an {\bf algebraic group}}, meaning that the algebra $\bk[G]$ is finitely generated.

\begin{definition}
    For an algebraic group $G$, define
    \[
    \Bbbk[G]^\circ:=\sum\limits_{I}(\Bbbk[G]/I)^*\,\subset\,\uHom(\bk[G],\unit),
    \]
    where $I$ runs over all cofinite ideals in $\Bbbk[G]$.  Further, set
    \[
    \Dist(G):=\sum\limits_{n\in\N}(\Bbbk[G]/\m_e^n)^*\subset\bk[G]^{\circ}.
    \]
    where $\m_e\sub\Bbbk[G]$ denotes the augmentation ideal.
\end{definition}

It is standard that $\Bbbk[G]^\circ$ has the structure of a cocommutative Hopf algebra in $\Ds$, and $\Dist(G)$ is a Hopf subalgebra.  If we consider the inclusion
\[G(\bk)=\Hom_{\CAlg}(\bk[G],\bk)\subset \bk[G]^{\circ},\;\, g\mapsto \delta_g,\]
then the elements $\delta_{g}$ span the Hopf subalgebra $\Bbbk G(\Bbbk)\sub\Bbbk[G]^\circ$.

 We now recall several facts proven in \cite[\S 3]{CS}.
  In the following, if $A,B$ are Hopf algebras such that $A$ is a left $B$-module algebra, then we write $A\#B$ for the smash product $A$ and $B$ (see \cite[Sec.~4.1]{M}).

\begin{lemma}[{\cite[Lem.~3.3]{CS}}]\label{lemma decom k[G] circ} 
    We have an isomorphism of Hopf algebras $\Bbbk[G]^\circ\cong\Dist(G)\#\Bbbk G(\Bbbk)$.  In particular, $\Dist(G)$ and $G(\Bbbk)$ generate $\Bbbk[G]^\circ$ as an algebra.
\end{lemma}

\begin{lemma}[{\cite[Prop.~3.4]{CS}}]\label{lemma nondeg pairing}
    The pairing $\Bbbk[G]^\circ\otimes\Bbbk[G]\to\mathbf{1}$ is nondegenerate.
\end{lemma}

In the following, recall that a full subcategory of an abelian category is called topologising if it is closed under taking subquotients and direct sums of objects.
\begin{cor}\label{cor rep G to dist ff}
    The natural functor $\Rep G\to \Mod_{\Bbbk[G]^\circ}$ is fully faithful, and realises $\Rep G$ as a topologising subcategory of $\Mod_{\Bbbk[G]^\circ}$.
\end{cor}

\begin{proof}
    Indeed, the coaction of $\Bbbk[G]$ is dual to the action of $\Bbbk[G]^\circ$, so the result follows quickly from Lemma \ref{lemma nondeg pairing} (see for instance, \cite[Sec.~7.17]{J}).
\end{proof}

\begin{cor}\label{cor G stable sub}
    Suppose that $G$ is an algebraic group, $U\in\Rep(G)$, and $V\sub U$ is any subobject in $\Ind\Ds$.  Then $V$ is a $G$-submodule if and only if $V$ is stable under $\Dist(G)$ and $G(\Bbbk)$.  
\end{cor}
\begin{proof}
The forward direction is clear.  For the converse, Lemma \ref{lemma decom k[G] circ} {implies} $V$ must be a $\Bbbk[G]^\circ$-submodule of $U$.  The conclusion then follows from Corollary \ref{cor rep G to dist ff}.
\end{proof}

Recall the notion of zero divisors from \S \ref{sec:prel}.

\begin{lemma}[{\cite[Lem.~3.5]{CS}}]\label{lemma conn zero div}
    Let $G$ be connected.  Then $f\in\Gamma(\Bbbk[G]):=\Hom_{\Ds}(
    \unit,\bk[G]    )$ is a zero divisor if and only if $f$ is nilpotent.
\end{lemma}

We now recall from \cite{CS} that if $G$ is an algebraic group and $H\sub G$ is a subgroup, then the homogeneous space $G/H$ exists as a scheme in $\Ds$, and we have $\Bbbk[G/H]=\Bbbk[G]^H$.

\begin{cor}\label{cor zero div on G/H}
    Let $G$ be an algebraic group, and $H\sub G$ a subgroup.  Then $f\in\Bbbk[G/H]$ is a zero divisor if and only if the restriction of $f$ to some connected component of $G/H$ is nilpotent.
\end{cor}

\begin{proof}
Recall from \cite[Lem.~5.1]{CS} that the quotient morphism $\pi:G\to G/H$ is faithfully flat, so by \cite[\href{https://stacks.math.columbia.edu/tag/0254}{Tag 0254}]{stacks-project}, the underlying topological space of the scheme $G/H$ has the quotient topology.  It follows that  $G^{\circ}/(G^{\circ}\cap H)$ must be a connected component of $G/H$, where $G^{\circ}$ is the identity component of $G$.  Further by left $G$-equivariance, all connected components of $G/H$ must be isomorphic to $G^\circ/G^\circ\cap H$. Therefore we may assume $G$ is connected.  Then $\Bbbk[G/H]=\Bbbk[G]^{H}\sub\Bbbk[G]$, and our result follows from Lemma \ref{lemma conn zero div}.
\end{proof}

\subsection{Algebras of fractions} We let $G$ be an affine group scheme in $\Ds$ and let $A$ be an arbitrary commutative algebra in $\Rep G=\Ind\mathrm{Rep}^fG$.

For what follows, notice that because $A$ is an algebra in $\Rep G$, the coaction morphism $a=a_A:A\to A\otimes\Bbbk[G]$ is an algebra morphism.  

\begin{lemma}\label{lemma coaction non zero div}
    The morphism $a_A$ takes non-zero divisors of $\Gamma(A)$ in $A$ to non-zero divisors of $\Gamma(A\otimes\bk[G])$ in $A\otimes\bk[G]$.
\end{lemma}
\begin{proof}
    Write $S:\Bbbk[G]\to\Bbbk[G]$ for the antipode on $\Bbbk[G]$, and let 
    \[
    T=(1\otimes m)\circ(1\otimes S\otimes 1)\circ(a_A\otimes 1):A\otimes\Bbbk[G]\to A\otimes\Bbbk[G].
    \]
    Note that $T$ is an algebra isomorphism.  The morphism $T$ corresponds to the pullback map for the morphism $G\times \Spec A\to G\times \Spec A$, $(g,x)\mapsto (g,g^{-1}x)$.  Thus we see that $T\circ a:A\to A\otimes\Bbbk[G]$ is the standard inclusion $A\to A\otimes\Bbbk[G]$ in the first factor. Thus for $f\in \Gamma(A)$ we have a commutative diagram:
    \[
    \xymatrix{
    A\otimes \Bbbk[G] \ar[rr]^{a_A(f)} \ar[d]_{T} && A\otimes\Bbbk[G] \ar[d]^{T}\\
    A\otimes\Bbbk[G]\ar[rr]^{f\otimes 1} && A\otimes\Bbbk[G]
    }
    \]
    Since the vertical arrows are isomorphisms, the result is now clear.
\end{proof}

Recall the construction $Q(-)$ from \S \ref{sec:prel}. Recall also that by convention $\Gamma(A)=\Hom_{\Ds}(\unit,A)$, so that $Q(A)$ is in general not an algebra in $\Rep G$, only in $\Ind\Ds$.
Lemma~\ref{lemma coaction non zero div} implies a commutative diagram:
\[
\xymatrix{
A \ar[rr]^{a} \ar[d] && A\otimes\Bbbk[G]\ar[d]\\
Q(A)\ar[rr]^{Q(a)} && Q(A\otimes\Bbbk[G]).
}
\]
Moreover, we can apply the above to the case $A=\bk[G]$, where the coaction is the comultiplication $\Delta$, showing that also $Q(\Delta)$ exists. Once existence is established, coassociativity is immediately inherited:
\begin{lemma}\label{lemma Q(a) coassoc}
    We have $(Q(a)\otimes 1)\circ Q(a)=(1\otimes Q(\Delta))\circ Q(a)$.
\end{lemma}

Now consider the ind-algebra $B=Q(a)^{-1}(Q(A)\otimes\Bbbk[G])$ in $\Ds$, and observe that $A\sub B$.  

\begin{lemma}\label{lemma largest ratl subalg}
The following hold:
\begin{enumerate}
    \item $B$ is a subalgebra of $Q(A)$;
    \item $Q(a)(B)\sub B\otimes\Bbbk[G]$, and this induces on $B$ the structure of a $G$-module algebra.
\end{enumerate}\end{lemma}

\begin{proof}
    That $B$ is an algebra is clear.  To obtain $Q(a)(B)\sub B\otimes \Bbbk[G]$, we use Lemma \ref{lemma Q(a) coassoc} to obtain that
    \[
    (Q(a)\otimes 1)Q(a)(B)=(1\otimes\Delta)Q(a)(B)\sub Q(A)\otimes\Bbbk[G]\otimes\Bbbk[G].
    \]
    Thus $Q(a)(B)$ lies in the largest subobject $T$ of $Q(A)\otimes\Bbbk[G]$ such that $(Q(a)\otimes 1)(T)\sub Q(A)\otimes\Bbbk[G]\otimes\Bbbk[G]$.  However, it is clear by definition of $B$ that $T=B\otimes\Bbbk[G]$. 

    To see that $B$ has the structure of a $G$-module algebra, we use Lemma \ref{lemma Q(a) coassoc} for coassociativity, and the counit axiom is also straightforward.
\end{proof}

\begin{example}\label{ex largest ratl subalg}
    Let $H\sub G$ be a subgroup, and consider the canonical $G$-action $a$ on $\bk[G]^H$. Then we have an inclusion of algebras in $\Rep G$:
    \[\bk[G]^H\;\subset\; B:=Q(a)^{-1}(Q(\bk[G]^H)\otimes\Bbbk[G]).\]
\end{example}

\subsection{Exact subgroups, beyond finite type}\label{sec:charobs}

We let $G$ be an affine group scheme in $\Ds$ and $H\sub G$ a subgroup.

	We denote by $G/_0H$ the presheaf on $\text{Aff}_{fpqc}$, the category of affine schemes with the fpqc topology, given by $A\mapsto G(A)/H(A)$, and write ${G/H}$ for its sheafification in this topology. When $G$ is algebraic, $G/H$ is representable by a scheme (more precisely, by \cite{CS}, it is representable in the fppf topology, and thus is so in the fpqc topology as well), so this notation is consistent with previous quotients. 
\begin{prop}\label{prop:GHaff}
    Assume that $G/H$ is represented by the affine scheme $\Spec \bk[G]^H$, then $H\sub G$ is exact.
\end{prop}
\begin{proof}
    This is a classical observation, going back to \cite[\S 4]{CPS}. Concretely, by \cite[Lemma~5.3]{CS}, $\bk[G]^H\to\bk[G]$ is faithfully flat and by the same lemma we have a canonical isomorphism $\bk[G]\otimes_{\bk[G]^H}\bk[G]\cong \bk[G]\otimes \bk[H]$. Hence it suffices to prove that
\[
\bk[G]\otimes_{\bk[G]^H}\Ind^G_H-\;\simeq\; \bk[G]\otimes-,
\]
since the right hand side is an exact functor.  To prove this claim, observe that
\[
\bk[G]\otimes_{\bk[G]^H}\Ind^G_HV=\Bbbk[G]\otimes_{\Bbbk[G]^H}(\Bbbk[G]\square^{\Bbbk[H]} V)\cong (\Bbbk[G]\otimes_{\Bbbk[G]^H}\Bbbk[G])\square^{\Bbbk[H]} V,
\]
where $\square^{\Bbbk[H]}$ denotes the cotensor product, the first equality is by definition, and the second follows from flatness of $\Bbbk[G]^H\to\Bbbk[G]$.  Applying the canonical isomorphism stated above then proves the claim.
\end{proof}

    Write $G=\lim\limits_{\leftarrow}G_{\alpha}$ for algebraic groups $G_{\alpha}$, and set $H_{\alpha}\sub G_{\alpha}$ so that $H=\lim\limits_{\leftarrow}H_{\alpha}$.  
	\begin{lemma}\label{lem:limsheaves}
		We have an isomorphism of (pre)sheaves ${G/H}\cong\lim\limits_{\leftarrow}{G_{\alpha}/H_{\alpha}}$.  
	\end{lemma}

	\begin{proof}
		We have natural map $G/_0H\to\lim\limits_{\leftarrow}{G_{\alpha}/H_{\alpha}}$ from the maps $G/_0H\to G_{\alpha}/_0H_{\alpha}\to{G_{\alpha}/H_{\alpha}}$.  We would like to prove this is a sheafification map.  It is injective by \cite[\S 5.3]{CS}, thus it suffices to prove local surjectivity.
		
		Let $a=(a_{\alpha})\in \lim\limits_{\leftarrow}{G_{\alpha}/H_{\alpha}}(A)$; then for each $\alpha$, choose a faithfully flat map $A\to B_{\alpha}$ such that $a_{\alpha}$ restricts to an honest coset $b_{\alpha}\in G_{\alpha}/H_{\alpha}(B_{\alpha})$.  Let $B$ denote the restricted tensor product of all $B_{\alpha}$.  Then we obtain an element $b=(b_{\alpha})\in\lim\limits_{\leftarrow}G_{\alpha}(B)/H_{\alpha}(B)$ which is the restriction of $a$ under the  faithfully flat map $A\to B$.  
		
		Now we would like to lift $b$ to an element of $G(B)/H(B)$; we may choose elements $g_{\alpha}\in G(B)$ and $h_{\alpha\beta}\in H_{\alpha}(B)$ such that $g_{\alpha}H_{\alpha}=b_{\alpha}$ and $\phi_{\beta\alpha}(g_{\beta})=g_{\alpha}h_{\alpha\beta}$.  Let us show we can find a faithfully flat cover $B\to C$ in which a compatible lift $g'=(g'_{\alpha})\in G(C)$ exists.
		
		We choose a well-order on $\{\alpha\}$ and use transfinite induction on the size of an ordinal $\lambda$ on which we can find a compatible lift. Clearly it can be done for $\lambda=1$. If it is true for $\lambda$ (and we denote the resulting replacement again by $g_\alpha,\alpha\le\lambda$), then we may do it for $\lambda+1$ by finding a cover $C_{\lambda}\to C_{\lambda+1}$ and $h\in H_{\lambda+1}(C_{\lambda+1})$ which restricts to $h_{\lambda,\lambda+1}\in H_{\lambda}(C_{\lambda+1})$, and then set $g_{\lambda+1}':=g_{\lambda+1}h^{-1}$.  
		
		Finally, if we can find a compatible lift for all ordinals less than some limit ordinal $\lambda$, then for all $\alpha<\lambda$ we have $\phi_{\lambda\alpha}(g_{\lambda})=g_{\alpha}h_{\alpha\lambda}$, and by the compatibility we necessarily have $\phi_{\beta\alpha}(h_{\beta\lambda})=h_{\alpha\lambda}$ for $\alpha<\beta$.  Thus the $h_{\alpha\lambda}$ form a compatible system, defining an element $h_{\lambda}\in H_{\lambda}(C_{\lambda})$, where $C_{\lambda}$ is the restricted tensor product of all $C_{\alpha}$ for $\alpha<\lambda$.  Now we set $g_{\lambda}':=g_{\lambda}h_{\lambda}^{-1}$, proving our induction step.
	\end{proof}

Following \cite[Definition~7.5.2]{C1}, we call $H$ {\bf observable} in $G$ if $\Ind^G_H$ is faithful. If $G$ is algebraic, then by \cite[Theorem~7.3.1]{C2} and the main result of \cite{CS}, an equivalent condition is that the scheme $G/H$ is quasi-affine, {\it i.e.} $G/H\to \Spec \bk[G]^H$ is an open immersion, see \cite[\S 5.4]{C2}.

	\begin{lemma}
	\label{thm simp meand ex}
    Assume that $G=\varprojlim G_\alpha$ and $H=\varprojlim H_\alpha$ for observable subgroups $H_\alpha\sub G_\alpha$.
		If $\Bbbk[G]^H$ is a simple algebra in $\Rep G$, then ${G/H}\cong\Spec\Bbbk[G]^H$ and hence $H$ is exact.
	\end{lemma}
\begin{proof}
    If $I\sub\Bbbk[G_{\alpha}]^{H_{\alpha}}$ is a non-zero, $G_{\alpha}$-invariant ideal, then $I\cdot\Bbbk[G]^H$ contains $1$, implying that there exists some $\beta>\alpha$ such that $I\cdot\Bbbk[G_{\beta}]^{H_{\beta}}$ contains $1$.  
	
	In particular, if $I_{\alpha}$ denotes the ideal cutting  out the complement of the open immersion $G_{\alpha}/H_{\alpha}\sub\Spec\Bbbk[G_{\alpha}]^{H_{\alpha}}$, we find that there exists $\beta$ such that the natural map $\Spec\Bbbk[G_{\beta}]^{H_{\beta}}\to\Spec\Bbbk[G_{\alpha}]^{H_{\alpha}}$ factors through $G_{\alpha}/H_{\alpha}$.  In particular, the map $\Spec\Bbbk[G]^H\to\Spec\Bbbk[G_{\alpha}]^{H_{\alpha}}$ factors through $G_{\alpha}/H_{\alpha}$ for any $\alpha$, implying we have a natural map:
	\[
	\Spec\Bbbk[G]^H\to\lim\limits_{\leftarrow}{G_{\alpha}/H_{\alpha}}\cong{G/H}.
	\]
	We also have a natural map in the opposite direction using that $\Spec\Bbbk[G]^H\cong\lim\limits_{\leftarrow}\Spec\Bbbk[G_{\alpha}]^{H_{\alpha}}$ and the natural map $G_{\alpha}/H_{\alpha}\to \Spec\Bbbk[G_{\alpha}]^{H_{\alpha}}$.  These maps are mutually inverse.
    So $H\sub G$ is exact by Proposition~\ref{prop:GHaff}.
\end{proof}

\begin{remark}
    \begin{enumerate}
        \item  A converse of Lemma~\ref{thm simp meand ex} is also true. Concretely, if $H\sub G$ is exact, then the algebra $\bk[G]^H$ is simple (by Lemma~\ref{lemma induced from exact is simple}) and $H$ is an inverse limit of observable algebraic subgroups by the argument at the start of \ref{452} below.
        \item The arguments in \ref{452}, together with \cite[Proposition~7.5.3]{C1}, also show that Lemma~\ref{thm simp meand ex} can be strengthened: If $H\sub G$ is observable and $\bk[G]^H$ is simple, then $H$ is exact. 
    \end{enumerate}
\end{remark}

\subsection{Proof of Theorem \ref{thm simple comm alg}} Let $G$ be an affine group scheme in $\Ds$ and let $A$ now be a simple commutative algebra in $\Rep G$. 

\subsubsection{Reduction to the case where $A^G=\Bbbk$ and $A$ has a $\Bbbk$-point.}\label{reduction}   Let $L$ be a field extension of the field $A^G$ from Lemma~\ref{lem:AG}. The following proof mimics \cite[Lem.~4.4]{Magid}.
    \begin{lemma}\label{lemma base change simple}
    $L\otimes_{A^G}A$ is a simple   $G_L=G\times_{\Bbbk}L$-algebra in $\Ind\Ds_L$.  
    \end{lemma}
    \begin{proof}
        Each basis $\{b_i:i\in I\}$ of $L$ over $A^G$ leads to a decomposition $L\otimes_{A^G}A\cong \bigoplus\limits_{i\in I}b_iA$.  Let $J$ be a non-zero $G_{L}$-stable ideal of $L\otimes_{A^G}A$.  Then for a compact subobject $X\sub J$ in $\Ds$, and a basis $\{b_i\}$ of $L$ over $A^G$, there are finitely many $i\in I$ for which $X\to b_iA$ is non-zero.  Let $X\sub J$ be a non-zero compact subobject and $\{b_i:i\in I\}$ be a basis for which the number of such $i$, call it $n$, is minimal (over all $X$ and all bases).  
        
         If $n=1$, one can quickly reduce to the situation $X\subset A$, which implies that $J=L\otimes_{A^G}A$.  If $n>1$, then up to reindexing we may write $X\sub b_1A\oplus\cdots\oplus b_n A$.  Write $Y_1\sub A$ for the image of the projection $X\to b_1A\cong A$, which is non-zero by assumption.   Now let $J'\sub A$ be the sum of subobjects $Y\sub A$ such that there exists a subobject $Z\sub J$ with $Z\sub b_1A\oplus\cdots\oplus b_n A$, and for which the projection of $Z$ onto $b_1A\cong A$ is $Y$. Clearly $Y_1\sub J'$ so that $J'\neq0$, and $J'$ is an ideal of~$A$.  Further, $J'$ is $G$-stable, for instance by Corollary \ref{cor G stable sub}, because it is stable under both $\Dist(G)$ and $G(\Bbbk)$. Thus $J'=A$, implying that there exists a compact subobject $Z\sub J$ with $Z\sub b_1A\oplus b_2A\oplus\cdots\oplus b_nA$ and such that under the first projection $Z\to b_1A$, the image is exactly $b_1\unit$.  For $U\sub\Dist^+(G)$, we have $U\cdot Z\sub b_2U(A)\oplus\cdots\oplus b_nU(A)$, and similarly for $g\in G(\Bbbk)$ we have $(\id-g)\cdot Z\sub b_2(\id-g)(A)\oplus\cdots\oplus b_n(\id-g)(A)$.  By our minimality assumptions it follows that $\Dist^+(G)\cdot Z=0$ and $G(\Bbbk)$ fixes $Z$, so by Corollary \ref{cor rep G to dist ff}, $Z\sub (L\otimes_{A^G}A)^G\cong L$.  Since $Z\neq0$, it follows that $J=L\otimes_{A^G}A$.
            \end{proof}
    
     Let $\varphi:A\to \Bbbk'$ be a closed point of $\Spec A$ (for $A$ interpreted as an algebra in $\Ind\Ds$). This exists since the (MN) condition implies that the only simple algebras in $\Ind\Ds$ are field extensions of $\bk$.

    \begin{lemma}\label{lemma orbit embedding}
        The orbit map at $\varphi$ induces an embedding of $G$-algebras $A\hookrightarrow\Bbbk'\otimes\bk[G]=\bk'[G_{\bk'}]$, and an embedding of fields $A^G\hookrightarrow \Bbbk'$.  
    \end{lemma}

    \begin{proof}
        This follows immediately from simplicity.
    \end{proof}

Let $\mathbb{L}'=\overline{\Bbbk'}$.  Clearly $\mathbb{L}=\overline{A^G}\sub \mathbb{L}'$.

\begin{lemma}\label{lem:extscal}
    To prove Theorem \ref{thm simple comm alg}, it suffices to show that there exists an exact subgroup $H_{\mathbb{L}'}\sub G_{\mathbb{L}'}$ such that $A\otimes_{A^G}\mathbb{L}'\cong\mathbb{L}'[G_{\mathbb{L}'}/H_{\mathbb{L}'}]$.
\end{lemma}

    \begin{proof}
        Using Lemma \ref{lemma base change simple}, it suffices to look at $A\otimes_{A^G}\mL$, and thus we may assume without loss of generality that $A^G=\Bbbk$.  If we show that $A\otimes_{\Bbbk}\mathbb{L}'\cong\mathbb{L}'[G_{\mathbb{L}'}/H_{\mathbb{L}'}]$ for an exact subgroup $H_{\mathbb{L}'}$, then $\Mod_A$ is the ind-completion of a pretannakian category and admits a symmetric tensor functor to $\Ds_{\mathbb{L}'}$.  By Theorem \ref{thm neutrality}, this implies that we have a tensor functor $F:\mod_A\to\Ds$. Moreover, by \cite[Theorem~7.3.1]{CEO}, the composite 
    \[\Rep G\xrightarrow{A\otimes -}\Mod_A\xrightarrow{F}\Ind\Ds \]
        is isomorphic (as a monoidal functor) to the functor forgetting the $G$-action. Hence, $F$ sends the multiplication morphism $A\otimes A\to A$ in $\Mod_A$ to a morphism $A\to \unit$ in $\Ind\Ds$. One verifies directly that this is an algebra morphism. In other words, we can now assume $\bk'=\bk$, meaning $\mathbb{L}'=\mathbb{L}$, which concludes the proof.   \end{proof}

    {\em Thus from now on}, by replacing $\Ds$ and $A$ with $\Ds_{\mathbb{L}'}$ and $ A\otimes_{A^G}\mathbb{L}'$, {\em we will assume that $A^G=\Bbbk$, and that $A$ has a $\Bbbk$-point $\varphi$.	}
	
	\subsubsection{Connection with exact subgroups}\label{452} The orbit map at $\varphi:A\to\bk$ defines an embedding $A\sub \Bbbk[G]$, by Lemma \ref{lemma orbit embedding}.   
    
    We can write $A=\bigcup\limits_{\alpha\in I} X_\alpha$ for compact subobjects $X_\alpha\in \Rep G$.  We view $I$ as a poset, where $\alpha\le\beta$	if $X_\alpha\sub X_{\beta}$.   Let $A_\alpha$ be the finitely generated subalgebra of $A$ generated by $X_\alpha$, and let $N_\alpha\subset G$ be the kernel of the representation $X_\alpha$, so that $G_\alpha:=G/N_\alpha$ is an algebraic group, and the action of $G$ on $A_\alpha$ factors through $G_\alpha$. 
	  
	Let $H_{\alpha}$ be the stabilizer of $\varphi$ in $G_{\alpha}$.  By \cite[Cor.~7.8]{CS}, the orbit map $G_{\alpha}/H_{\alpha}\to \Spec A_{\alpha}$ determined by the point $\varphi$ is an immersion (in the sense of \cite[Defn.~5.3.5]{C2}), and the image of the underlying topological space must be dense since the pullback map on functions remains injective.  It follows that the orbit map $G_{\alpha}/H_{\alpha}\to\Spec A_{\alpha}$ is an open immersion of a dense open subscheme. In particular $H_\alpha$ is an observable subgroup of $G_\alpha$ by \cite[Theorem~7.3.1]{C1}.

\begin{lemma}\label{lem:Aab}
    The embedding $A_{\alpha}\to\Bbbk[G_{\alpha}]^{H_{\alpha}}$ induces an isomorphism $Q(A_{\alpha})\xto{\sim}Q(\Bbbk[G_{\alpha}]^{H_{\alpha}})$.  For $\alpha<\beta$, we have a commutative diagram of inclusions:
	\[
	\xymatrix{
		A_{\alpha} \ar@{^{(}->}[r]\ar@{^{(}->}[d] & A_{\beta} \ar@{^{(}->}[d] \\ \Bbbk[G_{\alpha}]^{H_{\alpha}} \ar@{^{(}->}[r] & \Bbbk[G_{\beta}]^{H_{\beta}}
	}
	\]
\end{lemma}
\begin{proof}
    The isomorphism follows by Lemma  \ref{lemma open iso field frac} and Corollary \ref{cor zero div on G/H}. The commutative diagram is immediate.
\end{proof}

	\begin{lemma}\label{lemma below diagram}
		The inclusion $\Bbbk[G_{\alpha}]^{H_{\alpha}}\sub \Bbbk[G_{\beta}]^{H_{\beta}}$ takes non-zero divisors to non-zero divisors.
	\end{lemma}
	
	\begin{proof}
		 Let us first assume that $G_{\beta}$ is connected.  Since both $\Bbbk[G_{\alpha}]^{H_{\alpha}}$ and $\Bbbk[G_{\beta}]^{H_{\beta}}$ are subalgebras of $\Bbbk[G_{\beta}]$, we may apply Lemma \ref{lemma conn zero div}.
		
		For the more general case, notice that both homogeneous spaces $G_{\alpha}/H_{\alpha}$ and $G_{\beta}/H_{\beta}$ are a disjoint union of isomorphic, irreducible open subschemes.  Thus a function is a non-zero divisor if and only if it is so on every component.  Let $f$ be a non-zero divisor on $G_{\alpha}/H_{\alpha}$.  It suffices to choose one connected component $U$ of $G_{\alpha}/H_{\alpha}$ and show the pullback of $f$ is a non-zero divisor on every component of $V$ of $G_{\beta}/H_{\beta}$ which projects to $U$.  However for this we may take $U=G_{\alpha}^\circ/(H_{\alpha}\cap G_{\alpha}^\circ)$, i.e.~we can assume $G_{\alpha}$ is connected and replace $G_{\beta}$ by $\pi^{-1}(G_{\alpha}^\circ)$ under the natural quotient $\pi:G_{\beta}\to G_{\alpha}$.
		
		However, under our new hypotheses we already showed that $G_{\beta}^\circ/(H_{\beta}\cap G_{\beta}^\circ)\to G_{\alpha}/H_{\alpha}$ takes non-zero divisors to non-zero divisors.  By equivariance of the projection $G_{\beta}/H_{\beta}\to G_{\alpha}/H_{\alpha}$, a pullback will be a non-zero divisor on every component if it is on one, and this completes the proof.
	\end{proof}
	
	We have now shown that the two downward arrows and the bottom horizontal arrows in the diagram in Lemma \ref{lem:Aab} take non-zero divisors to non-zero divisors, implying that the inclusion $A_{\alpha}\to A_{\beta}$ takes non-zero divisors to non-zero divisors.  Thus we obtain an inclusion $Q(A_{\alpha})\sub Q(A_{\beta})$.   

    The following lemma is left as an exercise.

    \begin{lemma}\label{lemma exercise}
        Let $B$ be an algebra in $\Rep G$, and $\{B_{\alpha}\}$ a directed system of subalgebras of $B$ with $\bigcup\limits_{\alpha}B_{\alpha}=B$.  If each inclusion $B_{\alpha}\sub B_{\beta}$ takes non-zero divisors to non-zero divisors, then 
        \[
        Q(B)=\bigcup\limits_{\alpha}Q(B_{\alpha}).
        \]
    \end{lemma}

Denote the stabiliser of $\varphi$ in $G$ by $H$, so that $H=\varprojlim H_\alpha$ and $A$ is a subalgebra of $\bk[G]^H$.
	
	\begin{lemma}\label{lemma ratl fnts equal}
		$Q(A)=\bigcup_{\alpha} Q(A_{\alpha})=\bigcup_{\alpha}Q(\Bbbk[G_{\alpha}]^{H_\alpha})=Q(\Bbbk[G]^H)$.
	\end{lemma}
	
	\begin{proof}
		The first and last equality follow from Lemma \ref{lemma exercise}, and the middle equality is Lemma~\ref{lem:Aab}.
	\end{proof}

\begin{cor}\label{cor obs}
		We have $A=\Bbbk[G]^H$ for an exact subgroup $H$.
	\end{cor}
    \begin{proof}
    First note that with $H$ already defined above, if we can prove that the inclusion $A\subset\bk[G]^H$ is an equality, then exactness of $H\sub G$ follows from Lemma~\ref{thm simp meand ex}.

To prove this equality, write $a:\Bbbk[G]^H\to\Bbbk[G]^H\otimes\Bbbk[G]$ for the coaction. By Example \ref{ex largest ratl subalg}, we have algebra inclusions 
        \[A\,\sub\, \bk[G]^H \,\sub\, B:=Q(a)^{-1}(Q(\Bbbk[G]^H)\otimes\Bbbk[G])\]
        in $\Rep G$.  Let $X\sub B$ be a compact subobject in $\Rep G$, and let $J\sub A$ be the sum of the objects in $\{Y\sub A:YX\sub A\}$.  Then $J$ is an ideal of $A$ in $\Rep G$, and is nonempty because $Q(\Bbbk[G]^H)=Q(A)$ by Lemma \ref{lemma ratl fnts equal} which, by compactness of $X$ implies there is $f\in\Gamma(A)$ with $fX\subset A$.  Thus $J=A$, implying $X\sub A$.

        Hence we have shown that $B\subset A$, which thus implies the desired $\Bbbk[G]^H=A$.  
    \end{proof}

\subsubsection{Conclusion of the proof of Theorem \ref{thm simple comm alg}}

Corollary~\ref{cor obs} shows that every simple $G$-algebra $A$ with a $\bk$-point is of the form $\bk[G]^H$ for an exact subgroup $H<G$, which is sufficient for the proof of Theorem \ref{thm simple comm alg} by the reduction in \S \ref{reduction}.

\appendix

\section{A symmetric tensor category with finite dimensional morphism spaces but objects of infinite length}\label{App}

\subsection{Definitions}

We can generalise the notions of (artinian) tensor categories slightly as follows. A tensor category $\Cs$ over $\bk$ satisfying
\begin{enumerate}
    \item[(4')] $\dim_{\bk}\Hom(X,Y)<\infty$ for all $X,Y\in\cC$ 
\end{enumerate}
will be called {\bf Hom-finite}. As explained in Section~\ref{sec:prel}, it is well-known that artinian tensor categories are $\Hom$-finite.

An essentially small $\bk$-linear symmetric monoidal category $(\cC,\otimes,\unit)$ with $\bk$-linear tensor product is an {\bf ind-tensor category over $\bk$} if it satisfies
\begin{enumerate}
\item[(1')] $\Ind\cC$ is abelian ({\it i.e.} $\cC$ is ind-abelian);
\end{enumerate}
as well as (2) and (3) from \S\ref{prel:tc}. Obviously tensor categories are ind-tensor categories. Ind-abelian categories in general were studied in \cite{Sch}. The prototype of an ind-abelian category is the category of finitely presented modules over a ring.

We thus have increasingly general monoidal categories:
    $$\{\mbox{artinian tensor cat's}\}\subset\{\mbox{$\Hom$-finite tensor cat's}\}\subset \{\mbox{tensor cat's}\}\subset \{\mbox{ind-tensor cat's}\}.$$
    That the second inclusion is strict is well-known, see \cite[\S 2.9]{Del90}. We currently do not know if the third inclusion is strict.
    In \S\ref{App:Ex}, we show that the first inclusion is strict. In other words, while finite length objects imply finite dimensional morphism spaces, the converse is not true.

\subsection{An example}\label{App:Ex}

Throughout, let $\bk$ be a field of characteristic 0. We construct an example of a Hom-finite symmetric tensor category that is not artinian. Our construction requires the axiom of choice, but that should be possible to avoid.

For two non-negative integers $m< n$, consider the parabolic subgroup $P_{n,m}<GL_n$ of the general linear group over $\bk$ with Levi subgroup $GL_{n-m}\times GL_1^{\times m}$. Concretely, if we write $n\times n$-matrices in block form with $m+1$ blocks with the first block of size $n-m$, then $P_{n,m}$ is the subgroup of elements for which all blocks below the diagonal are zero. Hence $P_{n,m+1}<P_{n,m}$ and $P_{n,0}=GL_n$ while $P_{n,n-1}$ is a Borel subgroup.

We refer to \cite{Harman} for details on all of the constructions in this paragraph. Let $\UU$ be a non-principal ultrafilter on $\mZ_{>0}$. We will interpret it as an ultrafilter on $\mZ_{>m}\subset \mZ_{>0}$ as well, for any $m\in\mN$. We have the ultrapower $K:=\bk^{\UU}$, which can canonically be identified for $\UU$ as an ultrafilter on any $\mZ_{>m}$, and which is a field extension of~$\bk$. Furthermore $\prod_{\UU}\Rep_{\bk}P_{n,m}$ is a symmetric tensor category over $K$, where we keep $m\in\mN$ fixed and take the ultraproduct over $n\in\mZ_{> m}$. This tensor category will not be artinian, but we consider the tensor subcategory $\uRep P_{t,m}$ generated by the object
$$V[m]:=(V_{m+1},V_{m+2},V_{m+3},\cdots)\;\in\; \prod_{\UU}\Rep_{\bk}P_{n,m},$$
where $V_l$ is the natural representation of $GL_l$, restricted to the relevant parabolic. Furthermore, the symbol $t$ in the notation is the element 
$$(1,2,3,\cdots)\;\in\; K,$$
where we interpret $K$ as the ultrapower of $k$ over $\UU$ as an ultrafilter on $\mZ_{>0}$, which is thus the categorical dimension of $V[m]$. Now $\uRep P_{t,m}$ is artinian, which follows by the existence of the forgetful tensor functor to the pretannakian category $\uRep GL_{t-m}$ of \cite{Del07, EHS}, where the latter is similarly realised inside the tensor category $\prod_{\UU}\Rep_{\bk} GL_{n-m}$.

Finally, following a standard procedure as reviewed in \cite[\S 6]{CEO-A}, we define $\cC_t$ as the direct limit $\varinjlim \uRep P_{t,m}$ of the chain of tensor functors
$$\uRep GL_t= \uRep P_{t,0}\;\to\; \uRep P_{t,1}\;\to\;\uRep P_{t,2}\;\to\;\cdots.$$
The tensor functor from $\uRep P_{t,m}$ to $\uRep P_{t,m+1}$ is the ultraproduct of the restriction functors from $\Rep P_{n,m}$ to $\Rep P_{n,m+1}$, which are defined for all but one (namely $n=m+1$) factor in the product and hence induce a tensor functor on the ultraproduct.

\begin{prop}
    The symmetric tensor category $\cC_t$ over $K$ has finite-dimensional morphism spaces, but has objects of infinite length.
\end{prop}
\begin{proof}
    That the morphism spaces are finite dimensional follows from the observation that the tensor functor $\uRep P_{t,m}\to\uRep P_{t,m+1}$ is fully faithful, which follows from the corresponding claim for $\Rep P_{n,m}\to\Rep P_{n,m+1}$. The latter is indeed fully faithful, since the quotient variety $P_{n,m}/P_{n,m+1}$ is anti-affine (even projective), see \cite[Sec.~21.3]{Hum}, so that we can apply \cite[Theorem~4.14]{CEOP}.
    On the other hand, $\uRep P_{t,m}\to \uRep P_{t,m+1}$ sends $V[m]$ to $V[m+1]$. Since the length of $V[m]$ in $\uRep P_{t,m}$ is $m+1$, the image of $V[0]$ in $\cC_t$ has infinite length.
\end{proof}

\section{Neutrality of pretannakian categories over GR and MN categories}\label{app:Del}

In this section, we let $\bk$ be an algebraically closed field and $\Cs$ will be a fixed GR, MN pretannakian category over $\bk$. The following theorem is proved in \cite{Del02} and \cite{Duke} for $\cC=\sVec$ and in \cite{CEO-A} for $\cC=\Ver_p$, by a reduction to the case $\sVec$ via the Frobenius functor. The proof we present here is an immediate adaptation  of Deligne's proof for $\sVec$. In the current paper, the theorem is only used to improve Theorem~\ref{thm simple comm alg} from $\mathbb{L}$ being {\em some} field extension of $\bk$ to it being the algebraic closure of $A^G$. The theorem might be useful in the future for work on \cite[Conjecture~1.4]{BEO}.

\begin{thm}\label{thm neutrality}
    Suppose that $\Ds$ is a pretannakian category over $\bk$ and admits an exact $\bk$-linear symmetric monoidal functor $F:\Ds\to \Mod_R\Cs$, for some ind-algebra $R$ in $\Cs$.  Then there exists a symmetric tensor functor $\Ds\to\Cs$.
\end{thm}

Following Deligne's approach, the proof consists of two parts, dealt with in \S \ref{fingen} (generalising \cite[\S 4]{Del02}) and \S\ref{general} (generalising \cite[\S 6.4]{Duke}) below. Throughout the appendix we will refer to monoidal functors as $F$ in Theorem~\ref{thm neutrality} simply as `tensor functors'. We will abbreviate commutative ind-algebras in $\cC$ to `algebras' and commutative algebras in $\VEC$ to `$\bk$-algebras'.

\subsection{Finitely generated case}\label{fingen}

In this subsection, {\em we assume that the pretannakian category $\Ds$ is finitely generated} as a tensor category, meaning that there exists an $X\in\Ds$ such that every object in $\Ds$ is a subquotient of a polynomial in $X,X^\ast$.

For an algebra $R$ and two tensor functors $F,G:\Ds\to \Mod_R\cC$, we define the functor
$$\Isom_R(F,G):\;\CAlg_R\cC\to \Set$$
sending an $R$-algebra $R\to A$ to the set of (iso)morphisms between the tensor functors $\Ds\to \Mod_{A}\cC$ obtained from $F$ and $G$ via extension of scalars. So $\Isom_R(F,G)(R)$ is non-empty if and only if $F$ and $G$ are isomorphic.
\begin{lemma}\label{fintypeDel}
  The functor  $\Isom_R(F,G)$ is represented by a finitely generated and faithfully flat $R$-algebra.
\end{lemma}

\begin{proof} Except for the faithfully flat part, this is proved in \cite[Proposition~4.1]{Del02} for $\cC=\sVec$ and the proof carries over verbatim. In particular, the case $R=\bk$ (for general $\cC$) simplifies slightly and is written out in detail in the proof of \cite[Proposition~7.3.2]{CEO}. 

As in \cite[\S 3]{Del02}, we can realise the algebra as a non-zero direct limit of objects in the image of $\Ds\boxtimes\Ds\to\Mod_{R}\cC$. Objects in this image are rigid, so faithfully flat, implying that such direct limits are also flat. Moreover, such direct limits will contain a rigid (so faithfully flat) subobject (with flat quotient), making them faithful too.
\end{proof}

For tensor functors $F:\Ds\to\Mod_R\cC$ and $G:\Ds\to\Mod_S\cC$, for algebras $R,S$, we similarly define the functor
$$\Isom_{R\mid S}(F,G):=\Isom_{R\otimes S}(F\otimes S,R\otimes G):\CAlg_{R\otimes S}\cC\to \Set,$$
where $F\otimes S$ stands for the composition
\[\Ds\xrightarrow{F}\Mod_R\cC\xrightarrow{-\otimes S=-\otimes_RR\otimes S} \Mod_{R\otimes S}\cC.\]
We focus on the case $R=S$ and $F=G$, in which case \[I:=\Isom_{R\mid R}(F,F):\;\CAlg_{R\otimes R}\cC\to\Set\] can be interpreted as an affine groupoid acting on $\Spec R$, see \cite[3.10]{Del02}. More concretely, for any algebra $A$, the groupoid is the category with objects $\CAlg(R,A)$, and for any two such objects the set of morphisms between them is $I(A)$, with $A$ interpreted as an $R\otimes R$-algebra via the two objects. If $F$ were the extension of scalars along $\bk\to R$ of some tensor functor $F_0:\Ds\to\cC$, then clearly $I(R)\not=\varnothing$, although that is not how we will prove Theorem~\ref{thm neutrality} here. Rather, we will observe that $F$ lifts to a tensor functor taking values in the category of $I$-representations in $\cC$ and that the latter is equivalent to a category of representations in $\cC$ of a groupoid acting on an affine scheme of finite type, which will then imply Theorem~\ref{thm neutrality} since $\bk=\overline{\bk}$ implies such schemes must have a $\bk$-point. By Lemma~\ref{fintypeDel}, the defining algebra morphism $R\otimes R\to \bk[I]$ is such that $\bk[I]$ is finitely generated over $R\otimes R$, and hence finitely presented by the Hilbert Basis Theorem. We therefore obtain, as in \cite[Corollaire~4.3]{Del02}:
\begin{cor}\label{Corlim}
    If we write $R=\varinjlim_\alpha R_\alpha$, for algebras $R_\alpha$, then there is a groupoid $I_\alpha$ acting on $\Spec R_\alpha$, for some $\alpha$, such that $I$ is obtained from $I_\alpha$ via base change $R_\alpha\to R$.
\end{cor}

We say that a groupoid $G$ acting on an affine scheme $\Spec R$ in $\cC$ is {\bf transitive} if $G \to \Spec R \times \Spec R$ is an epimorphism in the category of fpqc-faisceaux.

\begin{lemma}\label{L44}\cite[Lemme~4.4]{Del02} Let $R\to S$ be a morphism in $\CAlg\cC$. If an affine groupoid $G$ acting on $\Spec S$ is obtained from a transitive groupoid $H$ acting on $\Spec R$ via extension of scalars, then extension of scalars induces an equivalence between the categories of representations of $H$ and $G$.
\end{lemma}
\begin{proof}
    The proof from \cite[\S 3]{Del90} carries over, so we only outline the argument. The groupoids $G$ and $H$ (which are both transitive) define pseudo-functors from the category of affine schemes in $\cC$ to the 2-category of (abstract) groupoids, which by construction are prestacks for the fpqc topology. Then one can observe that the category of $G$-representations is determined by this prestack and subsequently by the associated stack. The crucial observation is then that transitivity implies that the stacks associated to $G$ and $H$ are equivalent.
\end{proof}

\begin{example}
    Consider an algebra $R$ and the groupoid $H=\Spec R\times \Spec R$ acting (transitively) on $\Spec R$ in the obvious way. Then the category of $H$-representations is by definition the category of descent data for $\bk\to R$ (automatically faithfully flat), and hence equivalent to $\Ind\cC$. This agrees with the fact that $H$ is the extension of scalars along $\bk\to R$ of the trivial groupoid.
\end{example}

\begin{proof}[Proof of Theorem~\ref{thm neutrality} for $\Ds$ finitely generated]
    By taking a quotient (and applying extension of scalars), we can assume that $R$ is a $\bk$-algebra. We write this as a limit $\varinjlim_\alpha R_\alpha$ of finitely generated subalgebras. It thus follows from Corollary~\ref{Corlim} that $I$ is obtained via extension of scalars from a groupoid $I_\alpha$ acting on $\Spec R_\alpha$, for some $\alpha$.

    In order to be able to apply Lemma~\ref{L44}, we want $I_\alpha$ to be transitive. Slightly better, we will show that $R_\alpha\otimes R_\alpha\to \bk[I_\alpha]$ is faithfully flat. That $R\otimes R\to \bk[I]$ is faithfully flat follows from Lemma~\ref{fintypeDel}. Hence if we can show that $R_\alpha\otimes R_\alpha\to \bk[I_\alpha]$ is flat, it quickly follows that it is also faithful. Corollary~\ref{lem:Salpha} thus shows that $I_\beta$ (obtained via extension of scalars along $R_\alpha\to R_\beta$) is transitive, for some $\beta>\alpha$.

    Now, by construction of $I$, see \cite[\S 3.10]{Del02}, each $F(X)$, $X\in\Ds$, is an $I$-representation, and the structure is compatible with tensor products and natural (in other words $F$ lifts to the category of $I$-representations in $\cC$). By Lemma~\ref{L44}, $F$ is thus obtained by extension of scalars along $R_\beta\to R$ of a tensor functor $\Ds\to\Mod_{R_\beta}\cC$. But now $R_\beta$ is finitely generated, so taking a simple quotient, automatically isomorphic to $\bk$, yields the desired tensor functor $\Ds\to\cC$ via extension of scalars.
\end{proof}

\begin{remark}
    The arguments in this section actually show a slightly more general claim: Let $\mathscr{E}$ be pretannakian with a tensor functor $\mathscr{E}\to\cC$. Then, if there exists a tensor functor $\Ds\to\Mod_R\mathscr{E}$ (where $\Ds$ is still finitely generated pretannakian) for some $\bk$-algebra $R$, then there exists a tensor functor $\Ds\to\mathscr{E}$.
\end{remark}

\subsection{Tensor functors in terms of algebras}
We adapt slightly some notions from \cite[\S 3.3]{Duke} to deal with non-semisimple pretannakian categories as `base'. Throughout let $\cA$ and $\cB$ be arbitrary pretannakian categories over $\bk$.

\begin{definition}\label{def:Aalg}
    An {\bf $\cA$-algebra in $\cB$} is an ind-algebra $A$ in $\cB\boxtimes\cA$ such that
    \begin{enumerate}
        \item for every $X\in\cB$ there exists $Y\in\cA$ with $A\otimes (X\boxtimes\unit)\simeq A\otimes (\unit\boxtimes Y)$ as $A$-modules in $\Ind(\cB\boxtimes\cA)$, and
        \item $\Hom(\unit\boxtimes V,\unit)\;\xrightarrow{\Hom(\unit\boxtimes V,\eta)}\;\Hom(\unit\boxtimes V, A)$
    is an isomorphism, for every $V\in\cA$.
    \end{enumerate}

\end{definition}

For an $\cA$-algebra $A$ in $\cB$, we can consider the tensor functor
$$\Theta_A:\cA\to\mod_A(\cB\boxtimes\cA),\quad V\mapsto A\otimes(\unit\boxtimes V),$$
and observe that, almost by definition, it is fully faithful and essentially surjective, and hence an equivalence. This thus implies that $\cA$-algebras are absolutely flat and simple. We similarly have a tensor functor
$$\Omega_A:\cB\to\mod_A(\cB\boxtimes\cA),\quad X\mapsto A\otimes(X\boxtimes\unit).$$
\begin{thm}(\cite[Theorem~3.4.3]{Duke})\label{Thm:AlgFun}
The assignment $A\mapsto \Theta^{-1}_A\circ \Omega_A$ is an equivalence from the category of $\cA$-algebras in $\cB$ (as a full subcategory of the category of all ind-algebras in $\cB\boxtimes\cA$) to the category (groupoid) of tensor functors $\cB\to\cA$.
\end{thm}

\begin{proof}
    We can construct an inverse functor as follows. For a tensor functor $F:\cB\to\cA$ we consider
    $$\omega: \cB\boxtimes\cA\to\cA, \quad X\boxtimes Y\to F(X)\otimes Y,$$
    for which the right adjoint $\omega_\ast:\cA\to\Ind(\cB\boxtimes\cA)$ is exact. This follows for instance from \cite[Proposition~8.22]{Del90}. It thus follows, see \cite[Lemma~6.2.1]{CEO}, that $\cA$ is canonically equivalent to $\mod_B(\cB\boxtimes\cA)$ with $B=\omega_\ast(\unit)$ in a way that exchanges $\omega$ and $B\otimes-$. It follows that $B$ is an $\cA$-algebra and $F\mapsto B$, is the sought inverse.
\end{proof}

\begin{remark}\label{rem:AkG}
    If there exists a tensor functor $\cB\to\cA$, then by \cite[\S 8]{Del90} we have an equivalence $\cB\boxtimes\cA\cong \mathrm{Rep}^fG$ for some affine group scheme $G$ in $\cA$ and under this equivalence the associated $\cA$-algebra is sent to $\bk[G]$.
\end{remark}

\subsection{General case}\label{general}

Recall that $\cC$ is a GR+MN pretannakian category. Let $\Ds$ be pretannakian with a tensor functor $\Ds\to\Mod_R\cC$ for some ind-algebra $R$ in $\cC$.

\begin{lemma}\label{lem:reformulate} Assume that $\Ds$ is finitely generated.
    \begin{enumerate}
        \item $\Ds$ has a $\cC$-algebra, unique up to isomorphism.
        \item For a $\cC$-algebra $A$ in $\Ds$ and a tensor subcategory $\Ds'\subset\Ds$, denote by $A'$ the maximal subalgebra of $A$ that belongs to $\Ind(\Ds'\boxtimes\cC)\subset\Ind(\Ds\boxtimes\cC)$. Any algebra endomorphism of $A'$ lifts to one of $A$.
    \end{enumerate}
\end{lemma}
\begin{proof}
    Due to Theorem~\ref{Thm:AlgFun}, part (1) is a reformulation of the existence (proved in \S \ref{fingen}) and uniqueness, see \cite[Proposition~7.3.2]{CEO}, of tensor functors $\Ds\to\cC$.

    For part (2) we first observe that $A'$ is a $\cC$-algebra in $\Ds'$. Indeed, condition~\ref{def:Aalg}(2) is trivially inherited. That the first property is inherited we can observe via
    \[\Hom_A(A\otimes (X\boxtimes\unit),A\otimes (\unit\boxtimes Y))\;\cong\; \Hom( X\boxtimes Y^\vee,A)\;\cong\;\Hom( X\boxtimes Y^\vee,A'),\]
    for $X\in\Ds',Y\in\cC$, and the same observation for morphism in the other direction.
 Via Theorem~\ref{Thm:AlgFun} this then becomes a claim about lifts of endomorphisms of tensor functors $\Ds'\to\cC$ to $\Ds\to\cC$ which, by \cite[Th\'eor\`eme~8.17]{Del90} and \cite[Theorem~4.3.1]{CEO} becomes the claim that $G(\unit)\to G'(\unit)$ is surjective, for some affine group scheme $G$ in $\cC$ and quotient group $G'$. The latter is true by \cite[Porism~7.2.6]{C1}.
\end{proof}

\begin{lemma}\label{LemSlick}
Consider tensor subcategories $\Ds_1$ and $\Ds_2$ of $\Ds$ such that $\langle\Ds_1,\Ds_2\rangle=\Ds$, and $\Ds_2$ is finitely generated. 
If $A_1$ and $A_2$ are $\cC$-algebras of $\Ds_1$ and $\Ds_2$, respectively, then there exists an algebra morphism $A_{12}\to A_2$, where $A_{12}$ is the maximal subalgebra of $A_1$ in $\Ind((\Ds_1\cap\Ds_2)\boxtimes\cC)$. For any such morphism, the associated  algebra $A:=A_1\otimes_{A_{12}}A_2$ is a $\cC$-algebra for $\Ds$.
\end{lemma}

\begin{proof}
    The proof of the special case of $\cC=\sVec$ in \cite[Lemma~6.4.3]{Duke} carries over. The usage of \cite[Corollary~6.3.3]{Duke} can be replaced with that of \cite[Corollary~4.3.4]{CEO}. The usage of \cite[Theorem~6.3.1]{Duke} can be replaced with its generalisation \cite[Theorem~4.3.1]{CEO}. \cite[Lemma~6.4.2]{Duke} is to be replaced with Lemma~\ref{lem:reformulate} and \cite[Proposition~6.3.4]{Duke} by \cite[Corollary~7.2.7]{C1}.
\end{proof}

\begin{proof}[Proof of Theorem~\ref{thm neutrality} for general $\Ds$]
    We can now copy the proof of \cite[Theorem~6.4.1]{Duke}. In essence, we can construct $\cC$-algebras for all finitely generated subcategories of $\Ds$, by Lemma~\ref{lem:reformulate}, and by a transfinite induction argument using Lemma~\ref{LemSlick} we can take an appropriate limit over all of them to become a $\cC$-algebra for $\Ds$.
\end{proof}

\section{Extensions of rigid objects}\label{App:C}

In this appendix we give a self-contained proof of one version of the well-known principle that, under suitable hypotheses, an extension of two rigid objects is rigid, see also \cite[\S 4]{EP} and references therein.

In this appendix, we let $(\cB,\otimes ,\unit)$ be an abelian
monoidal category with right exact tensor product for which the functor
$\Hom(X\otimes-,Z)$
is representable, for all $X,Z\in\cB$, by an object $\uHom(X,Z)\in\cB$. We can interpret $\uHom(-,-)$ as a bifunctor $\cB^{\mathrm{op}}\times\cB\to\cB$.

We assume that every object in $\cB$ is a quotient of a left flat one (an object $X$ is left flat if $-\otimes X$ is an exact functor). Finally, we assume that we have an indexing set $S$, and for each $\alpha\in S$ an inverse system in $\cB$ with codirected poset $I_\alpha$ and objects $\{P^\alpha_i\mid i\in I_\alpha\}$, such that the associated functors
\[F_\alpha:=\varinjlim\Hom(P_i^\alpha,-):\;\cB\to\Ab\]
are exact and jointly faithful. This is naturally an assumption about the pro-completion of~$\cB$, but it will be more convenient to formulate things as above.

\begin{lemma}\label{lem:1}
    For an exact sequence $0\to X\to Y\to Z\to 0$ in $\cB$, with $Z$ right flat, and arbitrary $A\in\cB$, the induced sequence
    \[0\to X\otimes A\to Y\otimes A\to Z\otimes A\to 0\]
    is exact.
\end{lemma}
\begin{proof}
    The claim is tautological if $A$ is left flat. In general, we can take a resolution $P_2\to P_1\to P_0\to A\to 0$ with  left flat $P_i$ and the claim follows from diagram chasing. More conceptually, one can argue that we can construct an appropriate Tor-theory, and use $\Tor_1(Z,A)=0$.
\end{proof}


We call an object $X\in\cB$ left (resp. right) rigid if it has a left dual $X^\ast$ (resp. right dual~${}^\ast X$). In particular, if $X$ is left rigid, then it is right flat. 

\begin{lemma}\label{lem:2}
    For an exact sequence $0\to X\to Y\to Z\to 0$ in $\cB$ with $Z$ right rigid, and arbitrary $A\in\cB$, the induced sequence
    \[0\to \uHom(Z,A)\to \uHom(Y,A)\to\uHom(X,A)\to 0\]
    is exact.
\end{lemma}
\begin{proof}

    It suffices to show that
    \[0\to F_\alpha(\uHom(Z,A))\to F_\alpha(\uHom(Y,A))\to F_\alpha(\uHom(X,A))\to 0 \]
    is exact for every $\alpha$, or equivalently, that the direct limit of the sequences
    \[0\to \Hom(Z\otimes P_i^\alpha,A)\to \Hom(Y\otimes P_i^\alpha,A)\to\Hom(X\otimes P_i^\alpha,A)\to 0\]
    is exact.  Using that $\varinjlim$ preserves exact sequences (and left flatness of $P_i^\alpha$), it is thus sufficient to show that 
    $\varinjlim \Ext^1(Z\otimes P_i^\alpha, A)=0$ or equivalently that $\varinjlim \Hom(Z\otimes P_i^\alpha,-)$ is exact. The latter is clear, since the functor is $F_\alpha({}^\ast Z\otimes -)$, the composite of the exact functors ${}^\ast Z\otimes-$ and $F_\alpha$.
\end{proof}

\begin{prop}\label{prop:C}
  If for an exact sequence $0\to X\to Y\to Z\to 0$ in $\cB$ both $X$ and $Z$ are right rigid, then so is $Y$.
\end{prop}
\begin{proof}
    We use the criterion from \cite[Proposition~2.3]{Del90}, that the canonical natural transformation
    \[\xi_{B,A}:\uHom(B,\unit)\otimes A\to\uHom(B,A)\]
    is an isomorphism for all $A\in\cB$ if and only if $B$ is right rigid (with right dual $\uHom(B,\unit)$).

    We can consider the commutative diagram
    \[\xymatrix{
    0\ar[r]&\uHom(Z,\unit)\otimes A\ar[r]\ar[d]&\uHom(Y,\unit)\otimes A\ar[r]\ar[d]& \uHom(X,\unit)\otimes A\ar[r]\ar[d]&0 \\ 
    0\ar[r]&\uHom(Z,A)\ar[r]&\uHom(Y,A)\ar[r] &\uHom(X,A)\ar[r]&0 .}\]
By assumption, the left and right downward arrow are isomorphisms for all $A\in\cB$. By Lemma~\ref{lem:2}, the lower row is exact. Since $\Hom(X,\unit)\cong {}^\ast X$ is left rigid (it has left dual $X$) and thus right flat, Lemmas~\ref{lem:2} and~\ref{lem:1} show that the top row is also exact. It now follows that the middle downward arrow is also an isomorphism, so that $Y$ is right rigid indeed.
\end{proof}

\begin{remark}\label{rem:C}
    It follows from the same type of argument that extensions of left rigid objects are left rigid, if we instead assume that $\Hom(-\otimes Y,Z)$ is always representable, that every object is a quotient of a right flat one and if we have inverse systems of right flat objects with associated functors exact and jointly faithful.
\end{remark}

\begin{remark}
    We point out that for an object $X$ to admit a right dual, it is not sufficient for $X\otimes -$ to be left adjoint to a functor $X'\otimes -$ for some $X'\in\cC$, see for example \cite{HZ}.
\end{remark}

\subsection*{Acknowledgements}
The authors thank James Milne for pointing out the need to include a proof of Proposition~\ref{prop:FieldExt}.  They would also like to thank Pavel Etingof for helpful discussions. The research of the first author was supported by ARC grant FT220100125.  The research of the second author was supported by ARC grant DP210100251.

	\textsc{\footnotesize K.C.: School of Mathematics and Statistics, University of Sydney, NSW 2006, Australia} 
	
	\textit{\footnotesize Email address:} \texttt{\footnotesize kevin.coulembier@sydney.edu.au}
	
	\textsc{\footnotesize A.S.: School of Mathematics and Statistics, University of Sydney, NSW 2006, Australia} 
	
	\textsc{\footnotesize School of Mathematics and Statistics, UNSW Sydney, NSW 2052,
		Australia}
	
	\textit{\footnotesize Email address:} \texttt{\footnotesize xandersherm@gmail.com}


\begin{thebibliography}
	{DMNO}
	\bibitem[AM]{AM} M.~F. Atiyah and I.~G. Macdonald, {\it Introduction to commutative algebra}, Addison-Wesley Publishing Co., Reading, Mass.-London-Don Mills, Ont., 1969; MR0242802
	
	\bibitem[BEO]{BEO} D.~Benson, P.~Etingof, V.~Ostrik:
New incompressible symmetric tensor categories in positive characteristic.
Duke Math. J. 172 (2023), no. 1, 105--200.

\bibitem[BP]{BP} D.~Benson, J.~Pevtsova: Group schemes and their Lie algebras over a symmetric tensor category. arXiv:2507.02031.




\bibitem[CPS]{CPS}  E.~Cline, B.~Parshall, L.~Scott: Induced modules and affine quotients. Math. Ann. 230 (1977), no. 1, 1--14. 

\bibitem[C1]{Duke} Coulembier, Kevin. Tannakian categories in positive characteristic, Duke Math. J. {\bf 169} (2020), no.~16, 3167--3219; MR4167087.

    \bibitem[C2]{C1} Coulembier, Kevin. "Commutative algebra in tensor categories." Transform. Groups 31 (2026), no. 2, 1225--1272.

    \bibitem[C3]{C2} Coulembier, Kevin. "Algebraic geometry in tensor categories." arXiv preprint arXiv:2311.02264 (2023). To appear in Michigan Mathematical Journal.

    \bibitem[C4]{GR} Coulembier, Kevin. "On geometrically reductive tensor categories." arXiv preprint arXiv:2605.19167 (2026).

    \bibitem[CEO1]{CEO-A} K.~Coulembier, P.~Etingof, V.~Ostrik:
On Frobenius exact symmetric tensor categories.
With Appendix A by Alexander Kleshchev.
Ann. of Math. (2) 197 (2023), no. 3, 1235--1279.
	
	\bibitem[CEO2]{CEO} K.~Coulembier, P.~Etingof, V.~Ostrik:
Incompressible tensor categories.
Adv. Math. 457 (2024), Paper No. 109935, 65 pp.

	\bibitem[CEOP]{CEOP}K.~Coulembier, P.~Etingof, V.~Ostrik, B.~Pauwels:
Monoidal abelian envelopes with a quotient property.
J. Reine Angew. Math. 794 (2023), 179--214.
	\bibitem[CSZ]{CSZ} K.~Coulembier, M.~Stroiński, T.~Zorman. Simple algebras and exact module categories. arXiv:2501.06629.


    \bibitem[CS]{CS} Coulembier, Kevin, and Alexander Sherman. "Homogeneous spaces in tensor categories." arXiv preprint arXiv:2505.04848 (2025).

\bibitem[D1]{Del90} P.~Deligne: Cat\'egories tannakiennes. The Grothendieck Festschrift, Vol. II, 111--195, Progr. Math., 87, Birkh\"auser Boston, Boston, MA, 1990. 



	\bibitem[D2]{Del02} P.~Deligne: Cat\'egories tensorielles. Mosc. Math. J. 2 (2002), no. 2, 227--248.

    	\bibitem[D3]{Del07} P.~Deligne:
La cat\'egorie des repr\'esentations du groupe sym\'etrique $S_t$, lorsque t n'est pas un entier naturel. Algebraic groups and homogeneous spaces, 209--273,
Tata Inst. Fund. Res. Stud. Math., 19, Tata Inst. Fund. Res., Mumbai, 2007.

\bibitem[D4]{Del14}
P.~Deligne: Semi-simplicit\'e de produits tensoriels en caract\'eristique p. Invent. Math. 197 (2014), no. 3, 587--611.

    \bibitem[DM]{DM} P.~Deligne, J.S.~Milne: Tannakian Categories. In
Hodge cycles, motives, and Shimura varieties. 
Lecture Notes in Mathematics, 900. Springer-Verlag, Berlin-New York, 1982, pp. 101-228.

    \bibitem[EGNO]{EGNO} P.~Etingof, S.~Gelaki, D.~Nikshych, V.~Ostrik.
Tensor categories.
Math. Surveys Monogr., 205
American Mathematical Society, Providence, RI, 2015. 

\bibitem[EHS]{EHS}
I.~Entova-Aizenbud, V.~Hinich, V.~Serganova:
Deligne categories and the limit of categories~$Rep(GL(m|n))$.
 Int. Math. Res. Not. IMRN 2020, no. 15, 4602--4666.


\bibitem[EO]{EO} P.~Etingof, V.~Ostrik: On the Frobenius functor for symmetric tensor categories in positive characteristic. J. Reine Angew. Math. 773 (2021), 165--198.

\bibitem[EP]{EP} P.~Etingof, D.~Penneys:
Rigidity of non-negligible objects of moderate growth in braided categories.
Forum Math. Pi 14 (2026), Paper No. e7.


\bibitem[F]{Fo} T.J.~Ford:
Separable algebras.
Grad. Stud. Math., 183
American Mathematical Society, Providence, RI, 2017.


\bibitem[G]{EGA} Éléments de géométrie algébrique, par A. Grothendieck avec la collaboration de J. Dieudonn
é; EGA IV: Étude locale des schémas et des morphismes de schémas. La troisième
partie est Inst. Hautes Études Sci. Publ. Math., 28 (1966).


\bibitem[HZ]{HZ} S.~Halbig, T.~Zorman:
Duality in monoidal categories.
Math. Z. 313 (2026), no. 1, Paper No. 4, 35 pp.

\bibitem[Ha]{Harman} N.~Harman: Deligne categories as limits in rank and characteristic. arXiv:1601.03426.

\bibitem[HS]{HS} N.~Harman, A.~Snowden: Discrete pre-Tannakian categories. arXiv:2304.05375.

\bibitem[HSS]{HSS} N.~Harman, A.~Snowden, N.~Snyder: The Delannoy category.
Duke Math. J. 173 (2024), no. 16, 3219--3291.

\bibitem[H]{H} R. Hartshorne, {\it Algebraic geometry}, Graduate Texts in Mathematics, No. 52, Springer, New York-Heidelberg, 1977; MR0463157

\bibitem[Hu]{Hum} J.~E. Humphreys, {\it Linear algebraic groups}, Graduate Texts in Mathematics, No. 21, Springer, New York-Heidelberg, 1975; MR0396773.

\bibitem[J]{J} J.C.~Jantzen:
Representations of algebraic groups. 
Second edition. Mathematical Surveys and Monographs, 107. American Mathematical Society, Providence, RI, 2003.



\bibitem[La]{Lam} T.Y.~Lam:
A first course in noncommutative rings.
Grad. Texts in Math., 131
Springer-Verlag, New York, 1991.


\bibitem[Le]{Le} S.D.~Lentner: A conditional algebraic proof of the logarithmic Kazhdan-Lusztig correspondence. arXiv:2501.10735.

\bibitem[Ma]{Magid} A.R.~Magid,  Equivariant completions of rings with reductive group action.
J. Pure Appl. Algebra 49 (1987), no. 1-2, 173–185.

\bibitem[Mo]{M} S.~Montgomery, Hopf algebras and their actions on rings.
CBMS Regional Conf. Ser. in Math., 82, Providence, RI, 1993.

\bibitem[N]{N} Nastasescu, Constantin. "Théorème de Hopkins pour les catégories de Grothendieck." Ring Theory Antwerp 1980: Proceedings, University of Antwerp UIA Antwerp, Belgium, May 6–9, 1980. Berlin, Heidelberg: Springer Berlin Heidelberg, 2006. 88-93.

\bibitem[S]{Sch} D.~Sch\"appi:
Ind-abelian categories and quasi-coherent sheaves.
Math. Proc. Cambridge Philos. Soc. 157 (2014), no. 3, 391--423.

\bibitem[SY]{SY} K.~Shimizu, H.~Yadav: Commutative exact algebras and modular tensor categories. arXiv:2408.06314.

\bibitem[Stacks]{stacks-project} The Stacks Project Authors, \textit{Stacks Project}, \url{https://stacks.math.columbia.edu} (2018).

    
	\end{thebibliography}
\end{document}